\documentclass[11pt]{article}
\usepackage[utf8]{inputenc}
\usepackage{amsmath, amssymb, amsthm, dsfont, tikz, pgfplots, graphicx, float, mdframed, geometry, mathtools, xcolor, comment, caption, subcaption, array, enumerate, verbatim, url}
\usepackage[font=small]{caption}
\usetikzlibrary{math}
\pgfplotsset{compat=1.18}

\usepackage{hyperref}
\hypersetup{
    colorlinks=true,
    linkcolor=blue,
    citecolor=purple,
    urlcolor=blue,
    pdfborder={0 0 0}
}

\numberwithin{equation}{section}

\newtheorem{theorem}{Theorem}[section]
\newtheorem{corollary}[theorem]{Corollary}
\newtheorem{lemma}[theorem]{Lemma}

\newtheorem{proposition}[theorem]{Proposition}
\theoremstyle{remark}
\newtheorem*{remark}{Remark}
\newtheorem*{remark*}{Remark}
\theoremstyle{definition} 
\newtheorem{definition}[theorem]{Definition}
\newtheorem{problem}[theorem]{Problem}
\DeclareMathOperator*\lowlim{\underline{lim}}

\DeclareMathOperator*\uplim{\overline{lim}}

\let \le \leqslant
\let \ge \geqslant
\newmdtheoremenv{theorembox}{Theorem}
\newmdtheoremenv{lemmabox}{Lemma}
\def\cH{\mathcal{H}}
\def\cN{\mathcal{N}}

\def\eps{\varepsilon}

\title{Loewner--Kufarev entropy and large deviations of the Hastings--Levitov model}

\author{Nathana\"el Berestycki\footnote{University of Vienna, \texttt{nathanael.berestycki@univie.ac.at}}, Vladislav Guskov\footnote{KTH Royal Institute of Technology,  \texttt{vguskov@kth.se}}, Fredrik Viklund\footnote{KTH Royal Institute of Technology, \texttt{frejo@kth.se}}}
\date{\today}
 
\begin{document}
\maketitle
\begin{abstract}
We consider the Hastings--Levitov HL(0) model in the small particle scaling limit and prove a large deviation principle. The rate function is given by the relative entropy of the driving measure $\rho$ for the Loewner--Kufarev equation:
    \[
    \mathcal{H}(\rho) = \frac{1}{2\pi}\iint \bar{\rho}_t(\theta) \log \bar{\rho}_t(\theta) d\theta dt,
    \]
    whenever $\rho = \bar{\rho}_t d\theta dt/2\pi$ with $\int_{S^1} \bar{\rho}_t d\theta/2\pi = 1$.
    We investigate the class of shapes that can be generated by finite entropy Loewner evolution and show that it contains all Weil-Petersson quasicircles, all Becker quasicircles, a Jordan curve with a cusp, and a non-simple curve. We also consider the problem of finding a measure of minimal entropy generating a given shape as well as a simplified version of the problem for a related transport equation.
\end{abstract}

\tableofcontents 

\section{Introduction}
\subsection{The Hastings--Levitov model}

\subsubsection*{Description via iterated conformal maps} The Hastings--Levitov model was introduced \cite{hastings1998laplacian, carleson2001aggregation} in an attempt to describe stochastic Laplacian growth by means of a sequence of random conformal maps. Let $\mathbb{D}$ be the unit disc in the complex plane $\mathbb{C}$ and write $\Delta =\mathbb{\hat C}\smallsetminus \mathbb{\overline{D}}$ for the exterior disc. Consider the conformal map \[\varphi_d:\Delta\to\Delta\smallsetminus [1,1+d]\] fixing $\infty$ with positive derivative there. We write  $\varphi_{d, \theta}(z) = e^{i\theta}\varphi_{d}\left(e^{-i\theta}z\right)$ and think about $\varphi_{d, \theta}$ as attaching a ``particle'' (i.e., a slit) of size $d$ at an angle $\theta$ to the unit disk. Composing these maps gives a sequence $\{\Phi_n\}$ of conformal maps
\begin{align}\label{Phi_n_mapping}
\Phi_{n} = \varphi_{d_1, \theta_1}\circ\dots\circ\varphi_{d_n, \theta_n}.
\end{align}
The $\Phi_n$ maps the exterior disk to the exterior of a growing cluster $K_n$ whose shape will depend on the choice of angles $\{\theta_k\}$ and particle sizes $\{d_k\}$.

In the Hastings--Levitov model HL$(\alpha)$, $\alpha > 0$, the angles $\{\theta_k\}$ are chosen to be i.i.d. uniform random variables, which corresponds to attaching particles independently, according to harmonic measure seen from $\infty$. The particle sizes $\{d_k\}$ are given by 
\begin{align}\label{eq-HL-d-c}
d_k =  \varepsilon |\Phi_{k-1}'(e^{i\theta_k})|^{-\alpha/2}. 
\end{align}
Here $\alpha, \varepsilon>0$ are parameters, the latter  describing the overall scale of particles. 

Different values of the parameter $\alpha$  (heuristically) correspond to different lattice growth models. For instance, it has been argued that $\alpha =1$ corresponds to a continuum version of the Eden model while $\alpha = 2$ corresponds to a continuum version of diffusion-limited aggregation (DLA), see \cite{hastings1998laplacian} which also conjectures an intriguing phase transition at $\alpha =1$. In practice, one needs to regularize the model due to the difficulty in evaluating the derivative of the conformal map on the irregular cluster boundary, see, e.g., \cite{carleson2001aggregation, rohde2005some, johansson2015small}. In this paper we only consider the HL$(0)$ model which requires no regularization, and in this case the mapping \eqref{Phi_n_mapping} is a composition of i.i.d. random uniform rotations of the fixed conformal map $\varphi_{d}$. This model is the easiest to analyze and it has received substantial attention in the literature, see, \cite{norris2012hastings,rohde2005some, silvestri2017fluctuation, johansson2012scaling,berestycki2021explosive,johansson2009rescaled,berger2022growth}. 



\subsubsection*{Description via the Loewner-Kufarev equation}
It is useful to describe the HL$(0)$ growth process in continuous time using the Loewner--Kufarev equation, see \cite{carleson2001aggregation}. 
Briefly put, letting each slit particle grow continuously one after the other (keeping the overall order of the particles) and parametrizing appropriately, we obtain a family of exterior conformal isomorphisms $f_t : \Delta \to \Delta_t$, where $\Delta_t = \hat{\mathbb{C}} \smallsetminus K_t$ are decreasing simply connected domains and $f_t(z) = e^t z + o(1)$ as $z \to \infty$. The family $(f_t)_{t\ge 0}$ is known as a Loewner chain and satisfies the Loewner--Kufarev equation:
\begin{align}\label{eq:LK-eq}
  \dot f_{t}(z) =  zf'_{t}(z)p_t(z), \qquad  f_0(z) = z, \qquad z\in \Delta,
\end{align}
understood in the absolutely continuous (integrated) sense, where
\[
p_t(z) = \int_{0}^{2\pi}\frac{z+e^{i\theta}}{z-e^{i\theta}}\rho_t(d\theta)
\]
is an analytic function with positive real part and $p(\infty)=1$.
Here $(\rho_t)_{t\ge 0}$ is a family of probability measures on $S^{1}=[0,2\pi]$ which we shall often identify with the unit circle. 

Conversely, under suitable assumptions, a given family of probability measures $(\rho_t)$ (defined for a.e.\ $t$) can serve as an input, refereed to as driving measure, for the Loewner--Kufarev equation which can be solved to produce a Loewner chain $(f_t)$. The clusters $K_t = \hat{\mathbb{C}} \smallsetminus \Delta_t$ of the family of decreasing domains $\Delta_t := f_t(\Delta)$ is a family of growing compact sets. For example, if $\rho_t \equiv d\theta/2\pi$, we have $f_t(z) = e^t z$ and $K_t = e^t \overline{\mathbb{D}}$. In general, there is a one-to-one correspondence between growing clusters, Loewner chains, and driving measures. See, e.g., \cite{viklund_loewnerkufarev_2024} and the references therein. 

\begin{remark}
The most classical setup, used for instance when constructing Schramm-Loewner evolution curves, is that the measure is a continuously moving point mass on the circle: $\rho_t = \delta_{W_t}(\theta)$, where the continuous function $t \mapsto W_t$ takes values in $S^1$.
\end{remark}
Note that the Loewner--Kufarev equation \eqref{eq:LK-eq} ensures that time is parametrized in such a way that the cluster $K_t$ has $\log$-capacity $c(t) = \int_0^t |\rho_s| ds$ at time $t$, where $|\rho_s|$ is the total mass of $\rho_s$. In particular, if $\rho_s$ is a probability measure for every $s\ge 0$, then $K_t$ has capacity equal to $e^t$ for any $t\ge 0$, but one may also consider more general finite measures. In general, $f_t(z) = e^{c(t)}z + o(1)$ as $z \to \infty$.

The Hastings--Levitov HL($\alpha$) model $(K_1, K_2, \ldots, K_n)$ described above fits in the framework of Loewner chains. Reparametrizing time so that the corresponding cluster grown at time $t\ge 0$ (denoted with a small abuse of notation by $K_t$) has capacity given by $\text{cap}(K_t) = e^t$ for all $t\ge 0$, its family of driving measures $(\rho_t)_{t\ge 0}$ depends on $\varepsilon$ and we call it $\mu_\varepsilon = (\mu_{\varepsilon, t})_{t\ge 0}$. It is easy to check that 
\begin{align}\label{eq:mu_delta}
  \mu_{\varepsilon, t}(d\theta) = \sum\limits_{k=1}^{n}\delta_{\theta_{k}}(d\theta)\mathds{1}_{[T_{k-1}, T_k)}(t),
\end{align}
where $T_k = \sum_{i=0}^k t_i$, with  capacity increments $\{t_k\}$ given by
\begin{equation}\label{tk}
e^{t_k} = 1 + \frac{d_k^2}{4(1+d_k)}
\end{equation}
and $\{ d_k\}$ defined as in \eqref{eq-HL-d-c}. When the parameter $\alpha=0$, the slit length $d_k \equiv \varepsilon$ is constant. Also, $\delta_\theta$ denotes the Dirac measure at $\theta$. (Note that with this time-parametrization, $\mu_{\varepsilon,t}$ is a probability measure for each $0\le t < T_n$ and for all $\varepsilon>0$.) With this choice of parametrisation, the relation between the map $\Phi_n$ of \eqref{Phi_n_mapping} and the Loewner map $f_t$ is that $\Phi_n = f_{T_n}$. 

\subsubsection*{The small particle scaling limit}
We shall be concerned with the small particle scaling limit of HL$(0)$-clusters, see, e.g., \cite{norris2012hastings, johansson2012scaling, johansson2015small}. This is the limit where the particle size $\varepsilon \to 0$ and the number of particles $n\to +\infty$ simultaneously in such a way that the limiting cluster becomes non-trivial, in the sense that it has positive (and finite) log-capacity. That is, let $t_k$ be as in \eqref{tk} the individual log-capacity of each particle. (Recall that since $\alpha = 0$, we have $d_k = \varepsilon$ and thus $t_k = c(\varepsilon) := \tfrac{1}{4}\varepsilon^2 + o(\varepsilon)$.) 
For some given $T > 0$ we take $n = n(\varepsilon)$ such that $nc(\varepsilon) \to T$ as $\varepsilon  \to 0$. This ensures that, after $n$ particles have been added, the capacity of $K_n$ converges (in the limit as $\varepsilon \to 0$) to $e^T$. We are interested in the limiting cluster.

\subsection{Large deviation principles}
\subsection*{Large deviations in random geometry}
A large deviation principle (LDP) describes probabilities of rare events on an exponential scale. Let $(X_\varepsilon)_{\varepsilon > 0}$, be a family of random variables. Roughly speaking, $(X_\varepsilon)$ is said to satisfy an LDP if there exists a non-negative and lower semi-continuous function $I$, called rate function, and a function $r$, called the rate, with $r ( \eps ) \to \infty$ as $\eps \to 0$, such that 
\[
P(X_\varepsilon \in A) \approx e^{-r(\eps) \inf_{x \in A} I(x)}, \quad \varepsilon \to 0+.
\]
To make such a statement precise one has to specify the topology and replace $\approx$ by suitable estimates. A classical example is when $X_\varepsilon$ is taken to be $(\sqrt{\varepsilon}B_t)_{t\ge0}$, where $B$ is standard Brownian motion. In this case the rate function is the one-dimensional Dirichlet energy, which is finite for functions in the Sobolev space $W^{1,2}$ (the Cameron-Martin space); this is known as Schilder's theorem.
Recently there has been substantial interest in large deviation principles in random planar geometry, in particular in relation to Schramm-Loewner evolutions \cite{Wang2019_DeterministicLoewnerChainEnergy, APW, guskov2023large, krusell2024, abuzaid2025large}. Both the Loewner energy and Loewner--Kufarev energy (see below) arise as rate functions in this context. The study of these rate functions and the geometry of configurations (curves, loops, foliations) which are assigned finite energy has been  fruitful, see \cite{Wang2022survey} and the references therein. In this paper we study the Hastings--Levitov model from a similar perspective: we prove several LDPs and identify the corresponding rate functions. We then analyze the geometry of configurations that are assigned finite values by the rate function.
\subsection*{LDP for HL(0) in discrete time}
To state our first result we need some additional notation. We write $S^1 = [0, 2\pi)$ (recall also that we sometimes identify $S^1$ with the unit circle in the complex plane). For $T>0$, let $\mathcal{M}(S^1\times [0,T])$ be the space of Borel measures on the cylinder {$S^1 \times [0,T]$} and let
\begin{align*}
  \mathcal{N}_T = \{\mu\in \mathcal{M}(S^1\times [0,T]):\mu(S^1\times I) = |I| \text{ for any interval } I \subset [0, T]\},
\end{align*}
where $|I|$ denotes Lebesgue measure of $I$. When $T = +\infty$ we write $\mathcal{N}_+$. Elements of $\mathcal{N}_T$ will be referred to as (normalized) driving measures. By disintegration every $\rho \in \mathcal{N}_T$ corresponds to a family of probability measures $(\rho_t)$ as above.  We equip $\cN_T$ with the weak topology. We also view the driving measures $\mu_\varepsilon$ in \eqref{eq:mu_delta}
as a measure in $\mathcal{N}_T$ by setting $\mu_\varepsilon( d\theta, dt) = \mu_{\varepsilon, t} (d\theta) dt$. 


It is known that in the small particle limit (described above) as $\varepsilon\to 0 $, the HL$(0)$ driving measures for the Loewner--Kufarev equation \eqref{eq:LK-eq} converge towards the deterministic limit $\mu_{0}(d\theta,dt)=d\theta dt/2\pi$ almost surely, in the topology of weak convergence in $\mathcal{N}_T$. By continuity the corresponding Loewner chains also converge almost surely, in the Carath\'eodory sense. See~\cite[Theorem 2]{johansson2012scaling} for these facts. As pointed out above, the limiting driving measure $\mu_{0}$ generates growing discs $e^{t}\overline{\mathbb{D}}$. This statement can be interpreted as a shape theorem: the HL$(0)$ clusters converge  towards a family of growing concentric discs. 

In this article, we are interested in the large deviation event that the HL$(0)$ cluster process is ``close'' to a given growth process, potentially different from concentric discs, in the small particle limit. One way to make this precise is via the driving measures. The corresponding large deviation result is stated below, and is the first main result of this article. 


\begin{theorem}\label{theorem_LDP}
\begin{samepage}
The family of HL$(0)$ driving measures $\{\mu_\varepsilon:\varepsilon>0\}$ defined in \eqref{eq:mu_delta} satisfies an LDP in $\mathcal{N}_{T}$, equipped with the weak topology, with {rate $4/\varepsilon^2$}, and convex, good rate function $\cH:\mathcal{N}_{T}\to \mathbb{R}_+$ given by the  {relative entropy}:
    \begin{align}\label{LK-entropy}
     \cH(\rho) =   \frac{1}{2\pi}\iint_{S^1 \times [0,T]} \bar{\rho}_t(\theta) \log \bar{\rho}_t(\theta) d\theta dt
    \end{align}
    whenever $\rho = \bar{\rho}_t(\theta) d\theta dt/2\pi$ and the integral exists, and set to $+\infty$ otherwise.
\end{samepage}
\end{theorem} 

In other words, $\cH$ is a convex, lower semi-continuous function with compact level sets $\{ \cH( \rho) \le a \}$, such that for any open set $O$ and any closed set $C$ of $\mathcal{N}_T$, 
$$
\text{(i)\quad } \limsup_{\varepsilon \to 0} \frac{\varepsilon^2}{4} \log \mathbb{P} ( \mu_\varepsilon \in C) \le - \inf_C \cH( \rho) , 
$$
and 
$$
\text{(ii)\quad } \liminf_{\varepsilon \to 0} \frac{\varepsilon^2}{4} \log \mathbb{P} ( \mu_\varepsilon \in O) \ge - \inf_O \cH( \rho).
$$
 \begin{remark}
    The relative entropy functional arises from Sanov's classical theorem, see \cite{dembo2009large}, which is applied in the proof of Theorem~\ref{theorem_LDP}.
\end{remark}
Recall that, given two measures $\mu$ and $\nu$, such that $\nu$ is absolutely continuous with respect to $\mu$ with Radon-Nikodym derivative given by $f$, the relative entropy of $\nu$ with respect to $\mu$ (or Kullback--Leibler divergence) is the quantity
$$
H(\nu|\mu) : = \int \log f d\nu = \int f \log f d\mu. 
$$
Thus the rate function $\cH$ in the above theorem is the relative entropy with respect to the measure $\mu_0( d\theta,dt) = d\theta dt/2\pi $ on $S^1\times [0, T]$, whose total mass of its $t$ marginal is 1. That is, 
$$\cH(\rho) = H(\rho | \mu_0) = \iint_{S^{1}\times [0,T]}\bar\rho_{t}(\theta)\log\bar\rho_{t}(\theta) \mu_{0}(d\theta, dt),$$ 
where $\bar \rho_{t}(\theta)$ is the Radon-Nikodym derivative of $\rho$ with respect to $\mu_{0}$.
By Jensen's inequality, since $x\mapsto x \log x$ is convex on $(0, \infty)$ and $d\theta/2\pi$ is a probability measure, $\cH(\rho)\ge 0$. 

%


\begin{remark}
    The proof of Theorem \ref{theorem_LDP} also gives the corresponding result for the anisotropic HL(0) model of \cite{johansson2012scaling}. In this model the distribution of angles is chosen according to some given density $\lambda$ as opposed to being uniform. In this case, the rate function is
    \[
         \mathcal{H}_\lambda(\rho) =   \frac{1}{2\pi}\iint_{S^1 \times [0,T]} \rho_t^\lambda(\theta) \log \rho_t^\lambda(\theta) \lambda(\theta)d\theta dt
    \]
    and $\rho_t^\lambda$ is the Radon-Nikodym derivative of $\rho$ with respect to $\lambda(\theta)d\theta dt/2\pi$.
\end{remark}
\subsubsection*{LDP for HL(0) in continuous time}
 We now consider the Hastings--Levitov HL($0$) model, where particles arrive in continuous time according to a Poisson process rather than discrete time, see, e.g., \cite{johansson2009rescaled}. More precisely, let $(N(t), t \ge 0)$ be a Poisson  process with intensity $\lambda = 4/\varepsilon^2$, independent of the i.i.d. angles $\{ \theta_k\}_{k\ge 1}$ defining the discrete time model in \eqref{eq-HL-d-c}, \eqref{Phi_n_mapping}. We then set $\tilde \Phi_t: = \Phi_{N(t)}$ for $t\ge 0$.  We encode this growth process via a Loewner evolution $(\tilde f_t)_{t\ge 0}$. It is driven by a family of measures $\tilde \mu_\varepsilon = (\tilde \mu_{\varepsilon, t })_{t\ge 0}$ as in \eqref{eq:mu_delta}, but the form of these driving measures is slightly more complicated, as we are not simply reparametrizing by capacity: 
\begin{equation}
    \label{eq:tilde_mu}
    \tilde \mu_{\varepsilon, t} : = \sum_{k=1}^\infty \frac{t_k}{X_k} \delta_{\theta_k}(d\theta) 1_{(\tau_{k-1}, \tau_k]}(t).
\end{equation}
Here the $(X_k)_{k\ge 1}$ are the independent exponential random variables with parameter $\lambda$ defining the intervals of time during successive arrivals of the Poisson process $(N(t))_{t\ge 0}$, and for $k \ge 0$, $\tau_k = X_1+ \cdots + X_k$ is the time of the $k$th arrival in the Poisson process (so $N(\tau_k) = k $, and $\tau_k - \tau_{k-1} = X_k$). 

The coefficient of the Dirac mass in the right hand side of \eqref{eq:tilde_mu} ensures that the log-capacity of $\tilde K_t: = \mathbb{C} \setminus \tilde f_t(\Delta)$ grows by a total of $t_k$ (depending on $\varepsilon$ as in \eqref{tk}) during the interval between the time of arrival of the $(k-1)$th particle and that of the $k$th particle, as in \eqref{eq:mu_delta}. With these choices, for any $k\ge 0$ we have 
$$
\tilde \Phi_{\tau_k} = \tilde f_{\tau_k}
$$
and $\tilde K_t$ has capacity $\exp ( \lambda^{-1} N(t) + o(1))$, where $o(1)$ is a term which is uniformly bounded by a quantity tending to zero as $\varepsilon \to 0$ (and comes from the fact that in between the two successive values of $\tau_k$ the capacity grows by a fraction of $t_k$ and so is bounded by $t_k = O(\varepsilon^2)$). That is, for every fixed $t\ge 0$ we have  $\tilde f_t(z) = e^{c(t)}z  + o(1)$ as $z \to \infty$, where the log-capacity $c(t)$ of $\tilde K_t$ satisfies $c(t) = \lambda^{-1} N(t) + o(1)$, and $o(1)$ converges to zero as $\varepsilon \to 0$ uniformly.

Note that by the law of large numbers, as $\varepsilon\to 0$, $c(t) \to t$ in probability. But at the level of large deviations, it is perfectly possible for $c(t)$ to grow at a different rate. Denote
\begin{align*}
\tilde{\mathcal{N}} :
&= \big\{\mu\in \mathcal{M}(S^1\times \mathbb{R}): \text{ there exists a differentiable function $m: \mathbb{R} \to \mathbb{R}$ such that }\\
& \hspace{3.5cm} \mu(S^{1}\times I) = \int_I m_{t} dt \text{ for any interval } I \subset \mathbb{R}\big\}.
\end{align*}
 Given $\tilde \rho \in \tilde{\mathcal{N}}$, let $s\mapsto t(s)$ be an increasing diffeomorphism so that if we define $\rho$ via
\begin{equation}\label{eq:diffeo}
    \tilde \rho_s = t'(s) \rho_{t(s)},
\end{equation}   
 then $\rho$ has marginal mass $t\mapsto m_{t} \equiv 1$, i.e., $\rho \in \cN$ and the corresponding Loewner chains satisfy $f^{\tilde \rho}_s = f_{t(s)}^\rho$, see \cite{viklund_loewnerkufarev_2024}. 
To state 
our large deviation principle for this continuous-time HL$(0)$ model, we shall impose a restriction that the total number of particles (or equivalently the total capacity) is fixed, equal to what one would have if growing the shape at constant rate. 
Set
$$
\tilde{\cN}_T = \big\{ \tilde \rho \in \tilde{\cN}: \tilde m \text{ has support in $[0, T]$ and } \int_{0}^T \tilde m(s) ds = T \big\}
$$
and extend the definition of $\mathcal{H}$ to $\tilde{\mathcal{N}}_T$ by setting
\[
\mathcal{H}(\tilde \rho) = \iint_{S^1 \times [0,T]}\tilde{\rho}_s\log (2\pi \tilde{\rho}_s) d\theta ds.
\]
Finally, let
\[
h_T(f) = \int_0^T f \log f dt.
\]
\begin{proposition}\label{P:LDPcontinuous}
The family of continuous-time HL$(0)$ driving measures $\{\tilde \mu_\varepsilon:\varepsilon>0\}$ defined in \eqref{eq:tilde_mu} satisfies an LDP in $\tilde{\mathcal{N}_{T}}$, equipped with the weak topology, with {rate $ \lambda = 4/\varepsilon^2$} as $\varepsilon \to 0$, and convex, good rate function 
$$\tilde \rho \in \tilde{\cN}_T \mapsto \cH(\tilde \rho) = \cH( \rho) + h_T(\tilde m),
$$
where $m$ is an increasing diffeomorphism that relates $\tilde\rho$ and $\rho$ in the equation~\eqref{eq:diffeo}.
\end{proposition}
\noindent The proof is sketched in Section~\ref{SS:reparam}.

\subsubsection*{LDP for HL(0) shapes}
Theorem~\ref{theorem_LDP} and Proposition~\ref{P:LDPcontinuous} take into account the whole history of the growth process. It is also natural to consider only the shape of the cluster at a given time $T>0$, 
and ask for a corresponding large deviation event. One way to phrase such a result is in terms of the discrete time HL$(0)$-map $\Phi_{n}$ which induces a probability distribution on the space of normalized conformal maps
\begin{align*}
  \Sigma'_T = \big\{g \text{ univalent in $\Delta$, omits a neighborhood of $0$, and }g(z) = e^Tz + O(1/z) \big\},
\end{align*}
equipped with the locally uniform topology. 
\begin{corollary}\label{thm:potato}
    The maps $\Phi_n$ satisfy an LDP in $\Sigma'_T$ equipped with the topology of uniform convergence on compact subsets of $\Delta$, with rate $ 4/\varepsilon^2$, as $\varepsilon \to 0$, and good rate function given by
    \begin{equation}\label{eq:Hstar}
    \cH^*_T(\Phi) = \inf_{\rho}\cH(\rho),
    \end{equation}
    where $e^T$ is the capacity of $K=\hat{\mathbb{C}} \smallsetminus \Phi(\Delta)$, and the infimum is over driving measures generating $\Phi$ via the Loewner--Kufarev equation on $[0,T]$. 
\end{corollary}

\begin{proof}The result follows from Theorem~\ref{theorem_LDP} and the contraction principle \cite[Theorem 4.2.1]{dembo2009large}. Indeed, solutions to the Loewner--Kufarev equation depend continuously on their driving measures with respect to the weak topology and the locally uniform topology, respectively, see \cite{johansson2012scaling}.
\end{proof}
\begin{remark}
    The minimum in \eqref{eq:Hstar} is achieved by some measure in $\mathcal{N}_T$, see Lemma~\ref{lem:optimal-measure}.
\end{remark}
While the question --- how hard it is to build the shape $K$ using HL$(0)$ model --- is answered in principle by Corollary~\ref{thm:potato}, the answer is not very explicit. We are led to the following.
\begin{problem}\label{P:main}
    What are the shapes whose cost (rate function) is finite? And given such a shape, what growth processes achieve the infimum defining $\cH^*_T(\Phi)$ in \eqref{eq:Hstar}? 
\end{problem}

\subsection{Finite entropy shapes }
We are interested in describing the class of shapes that can be grown by a measure of finite entropy, where by shape we mean the boundary of a fixed compact hull, and by entropy we mean the large deviation functional $\cH$ appearing in Theorem~\ref{theorem_LDP}.
\subsubsection*{Loewner--Kufarev entropy and energy}
We begin by describing some parallels with the notion of Loewner--Kufarev  {energy}, see \cite{APW, viklund_loewnerkufarev_2024}, with a view of discussing shapes with finite entropy below. To simplify this comparison it is useful to consider a slightly more general setting, in which, instead of a family of growing hulls $(K_t)_{t\in [0, T]}$ of capacity $e^t$ for $t\in [0,T]$, and starting at time 0 from $K_0 = \mathbb{D}$, we have a growing family $(K_t)_{t\in \mathbb{R}}$ indexed by all of $\mathbb{R}$, and of capacity $e^t$ for $t\in \mathbb{R}$. In this manner, the hulls start at time $t = -\infty$ from a `single seed' (the origin), in the sense that $\cap_{t\in \mathbb{R}} K_t = \{ 0 \}$.

We thus introduce an enlarged space on which the Loewner--Kufarev entropy will be defined. 
We first define \[\mathcal{N} = \{\mu\in \mathcal{M}(S^1\times \mathbb{R}): \mu(S^{1}\times I) = |I| \text{ for any interval } I \subset \mathbb{R}\}.\]
Note that $\mathcal{N}_T$ can be embedded in $\mathcal{N}$ by extending $\mu \in \mathcal{N}_T$ to a measure in $\mathcal{N}$ by defining it to equal $d\theta dt/2\pi$ where not originally defined. We will always consider $\mathcal{N}_T$ embedded in $\mathcal{N}$ in this way.

A measure $\mu \in \cN$ defines a growing family of hulls $(K_t)_{t\in \mathbb{R}}$, known as a whole-plane Loewner chain, via the Loewner--Kufarev equation and the initial condition
$$\lim_{t\to -\infty} e^{t}f_t(z) = z,
$$
where $f_t: \Delta \to \Delta_t = \hat{\mathbb{C}} \smallsetminus K_t$ is the corresponding conformal map. As in \cite{viklund_loewnerkufarev_2024}, under regularity assumptions, one may think of $\partial K_t$ as defining a continuous foliation of $\mathbb{C} \smallsetminus \{0\}$. (However, as we will see, finite entropy measures do not in general give rise to foliations in the sense of \cite{viklund_loewnerkufarev_2024}.)
See Section~7.1 of \cite{viklund_loewnerkufarev_2024} for a detailed description of this setup. 


We are now ready to extend our definition of entropy. 

\begin{definition}
Let $\rho \in \mathcal{N}$ and write $\rho = \rho_t(\theta) d\theta dt = \bar{\rho}_t(\theta) d\theta dt/2\pi$ by disintegration. 
We define the {  Loewner--Kufarev entropy} of $\rho$ by 
\[
\cH(\rho) =  \frac1{2\pi}\iint_{S^1 \times \mathbb{R}}  \bar \rho_t(\theta) \log \bar \rho_t ( \theta) d\theta  dt
\]
whenever this is defined, and set $\cH(\rho) = +\infty$ otherwise. 
\end{definition}
\noindent Note that this definition is consistent with \eqref{LK-entropy} for a measure in $\mathcal{N}_T$ extended to a measure in $\mathcal{N}$ as explained above.


Next, for $f \in W^{1,2}(S^1)$ (absolutely continuous and derivative in $L^2$) let  \begin{align*}
\mathcal{D}(f) = \frac{1}{2}\int_{S^1}|f'(\theta)|^2d\theta
\end{align*}
be the one-dimensional Dirichlet energy.
\begin{definition}
Let $\rho \in \mathcal{N}$. 
The {  Loewner--Kufarev energy} of $\rho$ is defined by
\begin{align*}
    S(\rho) = \int_{\mathbb{R}}  \mathcal{D}(\sqrt{\rho_{t}})dt,
\end{align*}
whenever $\rho = \rho_t(\theta) d\theta dt$ and set to $\infty$ otherwise. We write $S_+$ for the integral over $\mathbb{R}_+$.
\end{definition}
The Loewner--Kufarev energy arises as the large deviation rate function for Schramm-Loewner evolution curves, SLE$_\kappa$, as the parameter $\kappa \to +\infty$. More precisely, an LDP in a similar setting as Theorem~\ref{theorem_LDP} was proved in \cite{APW} with rate function $S(\rho)$. That result was derived from an LDP for the local time of Brownian motion rather than from Sanov's theorem as in the case of HL$(0)$. 

The Loewner hulls generated by measures of finite Loewner--Kufarev energy were studied in depth in \cite{viklund_loewnerkufarev_2024}. These enjoy some remarkable properties that we summarize now. Let $\gamma$ be a Jordan curve and write $f:\mathbb{D} \to D,g:\Delta \to D^*$ for the conformal maps onto the bounded and unbounded component of $\mathbb{C} \smallsetminus \gamma$ respectively. Define the {  Loewner energy} of $\gamma$ by
\[
I^L(\gamma) = \frac{1}{\pi}\int_{\mathbb{D}}\left|\frac{f''}{f'}\right|^2 dz^2 + \frac{1}{\pi}\int_{\Delta}\left|\frac{g''}{g'}\right|^2 dz^2 + 4\log\frac{|f'(0)|}{|g'(\infty)|}.
\]
This quantity arises as the large devation rate function for SLE$_\kappa$, as $\kappa \to 0+$ \cite{Wang2019LoewnerEnergy, Wang2022survey}.
We say that $\gamma$ is a {Weil-Petersson quasicircle}\footnote{Recall that a quasicircle is the image of the unit circle under a quasiconformal homoemorphism of the plane.} if $I^L(\gamma) < \infty$. This interesting class of rectifiable quasicircles has many different characterizations, see \cite{bishop2025}. For now we just note that any $C^{3/2 + \epsilon}$, $\epsilon > 0$, Jordan curve belongs to this class and while no Weil-Petersson quasicircle has a corner, it may still fail to be $C^1$. 

It was proved in \cite{viklund_loewnerkufarev_2024} that the interfaces $(\partial \Delta_t)_{t \in \mathbb{R}}$ corresponding to a Loewner chain driven by $\rho$ with $S(\rho)< \infty$ form a family of Weil-Petersson quasicircles that continuously foliate $\mathbb{C} \smallsetminus \{0\}$ as $t$ ranges over the reals. We refer to such a collection of Jordan curves as a foliation (see \cite{viklund_loewnerkufarev_2024} Section~2 for more details and a more precise definition). In fact, given any Weil-Petersson curve $\gamma$ (separating $0$ from $\infty$), there exists a measure generating it as a shape and one obtains the Loewner energy by minimizing the Loewner--Kufarev energy,
    \[
    \inf S(\rho) = \frac{1}{16}I^L(\gamma)- \frac{1}{8}\log \frac{|f'(0)|}{|g'(\infty)|} ,
    \]
    where the infimum is taken over $\rho \in \mathcal{N}$ generating $\gamma$ as a shape. See Theorem~1.4 of \cite{viklund_loewnerkufarev_2024}. We will discuss the optimal $\rho$ below.

\medskip We now return to the problem of describing the class of shapes that have finite entropy. As we have seen, a given shape may be described in several ways, for instance by a measure $\rho$ generating it at a given time $T$, by the corresponding conformal map $\Phi$, or by the boundary of the cluster $\partial \Phi(\Delta)$. Given any such description, we slightly abuse notation and write \[\mathcal{H}^*(\rho)=\mathcal{H}^*(\Phi)=\mathcal{H}^*(\partial \Phi(\Delta)):=\inf_\mu \mathcal{H}(\mu)\] where $\mu$ ranges over those measures in $\mathcal{N}$ that generate the given shape. (We treat the special time $T$ for $\rho$ as implicit in the notation.)

We point out that both the Loewner energy $I^L$ (associated to a Jordan curve) and the minimal entropy $\cH^*$ (associated to a compact hull, whose boundary may or may not be a Jordan curve) are nonnegative functionals measuring how much a given shape deviates from being a circle (or disc). This raises the questions of how qualitatively different these two functionals are. The next result, based on the log-Sobolev inequality, gives the following relation.

\begin{theorem}\label{T:energy_entropy}
    Let $\rho \in \mathcal{N}$. Then
    \begin{align}
    \cH(\rho) \le 2 S(\rho).
\end{align}
In particular, if $K$ is any compact hull bounded by a Weil--Petersson quasicircle, then $K$ can be generated by a measure with finite Loewner--Kufarev entropy.
\end{theorem}
\noindent The result is proved in Section~\ref{SS:energyinequality}.



    Let $\mathcal{E}$ be the set of shapes (boundary of compact hulls $K$) that can be generated by Loewner--Kufarev evolution driven by a finite entropy measure. The theorem below summarizes some facts about $\mathcal{E}$ proved in this paper, the first of which follows immediately from Theorem~\ref{T:energy_entropy}. We define the class of Becker quasicircles in Section~\ref{sect:becker}.

\begin{theorem}\label{prop:shapes}
 Let $\mathcal{E}$ be as above. Then the following holds.
    \begin{enumerate}[(i)]
        \item $\mathcal{E}$ contains every Weil-Petersson quasicircle separating $0$ from $\infty$. 
         \item $\mathcal{E}$ contains every finite time Becker quasicircle. 
                \item $\mathcal{E}$ contains a Jordan curve with a corner.
        \item $\mathcal{E}$ contains a Jordan curve with a cusp.
        \item $\mathcal{E}$ contains a non-simple curve.
    \end{enumerate}
\end{theorem}
Since no quasicircle has a cusp, the theorem shows in particular that $\mathcal{E}$ contains a Jordan curve which is not a quasicircle. We do not know if $\mathcal{E}$ contains all quasicircles. The proof will be completed at the end of Section~\ref{section_finite_entropy}.


\subsection{Minimal entropy driving measures}
Given a shape with finite entropy, can one describe a measure of minimal entropy that generates it via the Loewner equation? Is such a measure unique? This optimization problem is reminiscent of an entropic optimal transport problem with a rather complicated constraint. Solving this in some generality appears difficult, and we do not answer these questions in this article. However, we give a counterexample that shows that a natural guess which is correct for the Loewner--Kufarev energy (given by the ``equipotentials'' of $K$) cannot in general be the answer to this optimization problem. 
In Section~\ref{section:problems} we study the minimal entropy functional $\cH^*(\Phi)$ (which can be viewed as a function of the shape $K$) along with several related optimization problems.

\subsection{Organization of the paper} The paper is organised as follows. Section~\ref{sect:prelim} contains basic properties of the entropy functionals. In Section \ref{section_proof} we give a proof of the large deviations theorem for HL$(0)$, Theorem~\ref{theorem_LDP}, and a sketch for the proof of the continuous-time LDP, Proposition~\ref{P:LDPcontinuous}. In Section \ref{section_finite_entropy} we discuss the proofs of the geometric properties:
Section \ref{SS:energyinequality} contains the short proof of the entropy-energy inequality, Theorem~\ref{T:energy_entropy}, Section~\ref{sect:becker} discusses Becker quasicircles, and we construct explicit examples of shapes with finite entropy in Section \ref{SS:examples}. Finally, in Section \ref{section:problems} we first recall from \cite{viklund_loewnerkufarev_2024} precisely how the equipotential foliation minimises the Loewner--Kufarev energy, and describe in Proposition \ref{problem:easy-energy} why that cannot be the case for the entropy. We also describe in Section \ref{SS:non-analytic} a non-analytic version of Loewner's equation and connections between the question of entropy minimization and a transport equation.  

\subsection*{Acknowledgements} N.B.  acknowledges the support from the Austrian Science Fund (FWF) grants 10.55776/F1002 and 10.55776/PAT1878824. V.G. acknowledges support from the Swedish Research Council and the Göran Gustafsson Foundation for Research in Natural Sciences and Medicine. F.V. acknowledges support from the Knut and Alice Wallenberg Foundation, the Göran Gustafsson Foundation for Research in Natural Sciences and Medicine, and the Simons Foundation. 

\medskip We thank Ellen Krusell, Alan Sola and Yilin Wang for discussions and comments on a draft of the paper, and Vittoria Silvestri for discussions in early phases of this project.

\section{Preliminaries}\label{sect:prelim}
\subsection{Basic properties of the entropy}
The following proposition summarizes some easy facts about the entropy $\mathcal{H}$ defined in \eqref{LK-entropy}. We recall that
\begin{align*}
\tilde{\mathcal{N}} :
&= \big\{\mu\in \mathcal{M}(S^1\times \mathbb{R}): \text{ there exists a differentiable function $m: \mathbb{R} \to \mathbb{R}$ such that }\\
& \hspace{3.5cm} \mu(S^{1}\times I) = \int_I m_t dt \text{ for any interval } I \subset \mathbb{R}\big\}
\end{align*}
and $\mathcal{N}$ is the subset of $\tilde{\mathcal{N}}$ where $m_t  \equiv 1$.
\label{SS:reparam}
\begin{proposition}
    \label{lem:entropy-facts}
    Let $\rho \in \tilde{\mathcal{N}}$. Then the following holds.
    \begin{enumerate}[(i)]
    \item{
    If $\rho \in \mathcal{N}$, then $\cH(\rho) \ge 0$ with equality if and only if $\rho \equiv d\theta dt/2\pi$;}
    \item{$\cH$ is lower semi-continuous on $\cN$;}
    \item{Let $s \mapsto t(s)$ be an increasing diffeomorphism of $\mathbb{R}$.  If $\tilde \rho$ is the time-changed measure \[\tilde \rho_s = t'(s) \rho_{t(s)},\] then 
\begin{equation}\label{eq:Hreparam}
\cH(\tilde \rho ) = \cH(\rho) + h_{\mathbb{R}}(\tilde m) - h_{\mathbb{R}} ( m),
\end{equation}
whenever $h_{\mathbb{R}}( \tilde m) $ and $h_{\mathbb{R}} ( m) $ are well-defined, where 
\[
h_{\mathbb{R}}(f) = \int_\mathbb{R} f \log f dt
\]
is the differential entropy of $f:\mathbb{R} \to \mathbb{R}$ on $\mathbb{R}$.
}
    \end{enumerate}
\end{proposition}
\begin{proof}[Proof of Proposition \ref{lem:entropy-facts}]
$(i)$ This follows immediately from Jensen's inequality and the fact that $x\mapsto x \log x$ is convex. 

$(ii)$ Let $\rho \in \mathcal{N}$. Recall that $h_{S^1}(\bar{\rho}_t)= \tfrac1{2\pi} \int_{S^1} \bar{\rho}_t \log \bar{\rho}_t d\theta$ is the relative entropy of $\bar{\rho}_t d\theta$ with respect to $d\theta / 2\pi$. Since $x\log x  + e^{-1}$ is lower semi-continuous and non-negative it follows from Fatou's lemma that $h_{S^1}$ is lower semi-continuous.

Moreover, since $\cH( \rho) = \int_{\mathbb{R}} h_{S^1} ( \bar \rho_t) dt$ and $h_{S^1}(\bar \rho_t)  \ge 0$ (this is where we use the assumption that $\rho \in \mathcal{N}$), we conclude that $\cH$ is lower semi-continuous on $\mathcal{N}$ by another application of Fatou's lemma, as desired.


$(iii)$ Suppose that $\tilde \rho \in \tilde{\mathcal{N}}$ and that $h_{\mathbb{R}} (\tilde m)$ as well as $h_{\mathbb{R}} (\tilde m)$ are well defined. 
We have 
\begin{align*}
\cH(\tilde{\rho}) & = \iint_{S^1 \times \mathbb{R}}\tilde{\rho}_s\log (2\pi \tilde{\rho}_s) d\theta ds \\
& =  \iint_{S^1 \times \mathbb{R}}\rho_{t(s)} \log(2\pi t'(s)\rho_{t(s)})  t'(s) d\theta  ds\\
& = \iint_{S^1 \times \mathbb{R}}\rho_{t} \log (2 \pi\rho_{t}) d\theta dt + \iint_{S^1 \times \mathbb{R}}\rho_{t(s)} t'(s) \log t'(s)  d\theta ds  \\
& = \cH(\rho) + \int_{ \mathbb{R}}m_{t(s)} t'(s) \log t'(s)  ds \\
& = \cH(\rho) + h_{\mathbb{R}} ( \tilde m) - \int_{\mathbb{R}} m_{t(s)} t'(s) \log m_{t(s)} ds\\
& = \cH(\rho) +h_{\mathbb{R}}(\tilde m) - h_{\mathbb{R}}( m),
\end{align*}
where we used that $\int_{S^1}\rho_t(\theta) d\theta =m_t$ for a.e. $t\in \mathbb{R}$.
\end{proof}
\begin{remark}
    Note that the differential entropy $h$ can be both positive and negative and that $\tilde \rho_s$ is not necessarily a probability measure.
\end{remark}

Proposition \ref{lem:entropy-facts} shows that $\cH$ is not invariant under time-reparametrisation. But it is, in fact, easy to modify the definition of $\cH$ to obtain a notion which \emph{is} independent of the time-parametrisation.

\begin{corollary}
    Let $\rho = \rho_t d\theta dt\in \tilde{\mathcal{N}}$ with mass $m_t = | \rho_t|$, $t\in \mathbb{R}$. Define the  {invariant entropy} as
    $$\cH_{\mathrm{inv}}(\rho) = \cH(\rho) - h_{\mathbb{R}} (m).
    $$
    Then if $\tilde \rho$ is as in \eqref{eq:diffeo}, $\cH_{\mathrm{inv}}(\tilde \rho) = \cH_{\mathrm{inv}}( \rho)$. 
\end{corollary}
While the above invariant entropy may be more natural to work with from a geometric point of view, we do not know of a probabilistic interpretation.


\begin{remark}
Theorem~\ref{T:energy_entropy} gives an upper bound on the entropy in terms of the Loewner--Kufarev energy. Pinsker's inequality can be used to give a lower bound:
\begin{align}\label{eq:pinsker}
    H(\nu| \mu) \ge  \frac{1}{2}\left(\int_\Omega|f-1| \, d\mu\right)^2,\text{ where }f =\frac{d\nu}{d\mu}.
\end{align}
See \cite[Theorem 2.12.24]{bogachev2007measure}. This uses that we are considering probability measures.
\end{remark}
\begin{lemma}   \label{lem:optimal-measure}
Suppose $\Phi$ is generated by a Loewner chain with $\rho \in \mathcal{N}_T$ and that $\mathcal{H}(\rho) < \infty$. Then there exists a measure $\rho_\infty \in \mathcal{N}_T$ such that 
\[
\mathcal{H}_T(\rho_\infty) = \mathcal{H}_T^*(\Phi) 
\]
\end{lemma}
\begin{proof}
Let $\rho_n \in \mathcal{N}_T, n = 1,2, \ldots,$ be a sequence such that each $\rho_n$ generates $\Phi$ at time $T$ and $\lim_{n \to \infty} \mathcal{H}_T(\rho_n) = \mathcal{H}_T^*(\Phi)$. Since $S^1 \times [0,T]$ is compact, there is a further subsequence $\rho_{n_k}$  which converges to some measure $\rho_\infty$ as $k \to \infty$. By continuity of the Loewner map $\rho \mapsto \Phi$, this limiting measure generates $\Phi$ as well so $\mathcal{H}_T(\rho_{\infty}) \ge \mathcal{H}_T^*(\Phi)$. On the other hand, by lower semicontinuity,
\[
\mathcal{H}_T^*(\Phi) = \lim_{k\to \infty} \mathcal{H}_T(\rho_{n_k}) \ge \mathcal{H}_T(\rho_{\infty}).
\]
The claim follows.
\end{proof}



\section{Proofs of the LDPs}\label{section_proof}
\subsection{Proof of Theorem~\ref{theorem_LDP}}
The proof of Theorem~\ref{theorem_LDP} will be completed at the end of this section.

The space $\mathcal{M}_{+}(S^1 \times [0,T])$ of non-negative finite measures, as well as the space $\mathcal{M}_{1}(S^1 \times [0,T])$ of probability measures, equipped with the weak topology, is metrizable, \cite[Theorem 8.9.4]{bogachev2007measure}.  For our purposes the metric given by the Kantorovich–Rubinstein distance is convenient. It is generated by the Kantorovich–Rubinstein norm:
\begin{align*}
    ||\mu||_{\text{KR}} = \sup\left\{\mu(f): f\in\text{Lip}_{1}(S^1 \times [0,T]),\ \sup_{x\in S^1 \times [0,T]}|f(x)|\le 1\right\},
\end{align*}
where
\begin{align*}
    \text{Lip}_{1}(S^1 \times [0,T]) = \left\{f:S^1 \times [0,T]\to\mathbb{R}, |f(x)-f(y)|\le|x-y|\  \forall x,y\in S^1 \times [0,T]\right\}.
\end{align*}
Note that $||\cdot||_{\text{KR}}\le||\cdot||_{\text{TV}}$, where the latter is the total variation norm. The Kantorovich–Rubinstein distance, given by
\begin{align*}
    d_{\text{KR}}(\mu,\tilde\mu) = ||\mu-\tilde\mu||_{\text{KR}},
\end{align*}
metrizes the weak topology both in $\mathcal{M}_{+}(S^1 \times [0,T])$ and in $\mathcal{M}_{1}(S^1 \times [0,T])$, \cite[Theorem 8.3.2]{bogachev2007measure}. Moreover, we may metrize $\mathcal{M}_{1}(S^{1})$ in precisely the same way, and use the same notation $d_{\text{KR}}$ for the metric.

For a measure $\mu\in\mathcal{M}_{+}(S^1 \times [0,T])$, denote by $\langle\mu\rangle_{m}$ the approximation of $\mu$, obtained by averaging over intervals of length $1/m$, 
\begin{align*}
\langle\mu\rangle_{m} (d\theta, dt)= \sum\limits_{i=1}^{m}\left(m\int_{\frac{i-1}{m}}^{\frac{i}{m}}\mu(d\theta, ds)
    \right)\mathds{1}_{\left[\frac{i-1}{m},\frac{i}{m}\right)}(t)dt.
\end{align*}
\begin{lemma}\label{lemma:approximation}
    For any $\mu\in\mathcal{M}_{+}(S^1 \times [0,T])$, $d_{\text{KR}}(\mu,\langle\mu\rangle_{m})\le \frac{\mu(S^1 \times [0,T])}{m}$. In particular, the approximation $\langle\mu\rangle_{m}$ converges weakly to $\mu$ as $m\to\infty$.
\end{lemma}
\begin{proof}
Let $f\in\text{Lip}_{1}(S^1 \times [0,T])$ with $\sup_{(\theta,t)\in S^1 \times [0,T]}|f(\theta,t)|\le 1$.
The difference $\mu(f)-\langle\mu\rangle_{m}(f)$ can be estimated as follows:
\begin{align*}
    \left|\mu(f)-\langle\mu\rangle_{m}(f)\right| &= \left|\sum\limits_{i=1}^{m}\left(\int_{0}^{2\pi}\int_{\frac{i-1}{m}}^{\frac{i}{m}}f(\theta,t)\mu(d\theta,dt) - \int_{0}^{2\pi}\int_{\frac{i-1}{m}}^{\frac{i}{m}}f(\theta,\tilde t) m\int_{\frac{i-1}{m}}^{\frac{i}{m}}\mu(d\theta, dt) d\tilde t\right)\right|\\
    &= \left|m\sum\limits_{i=1}^{m}\left(\int_{0}^{2\pi}\int_{\frac{i-1}{m}}^{\frac{i}{m}}f(\theta,t)\mu(d\theta,dt)\int_{\frac{i-1}{m}}^{\frac{i}{m}}d\tilde t - \int_{0}^{2\pi}\int_{\frac{i-1}{m}}^{\frac{i}{m}}f(\theta,\tilde t) \int_{\frac{i-1}{m}}^{\frac{i}{m}}\mu(d\theta, dt) d\tilde t\right)\right|\\
    &\le m\sum\limits_{i=1}^{m}\int_{0}^{2\pi}\int_{\frac{i-1}{m}}^{\frac{i}{m}}\int_{\frac{i-1}{m}}^{\frac{i}{m}}\underbrace{\left|f(\theta, t) - f(\theta, \tilde t)\right|}_{\le|t-\tilde t|\le \frac{1}{m}}\mu(d\theta,dt)d\tilde t \\
    &\le \frac{\mu(S^1 \times [0,T])}{m}.
\end{align*}
Taking the supremum over all such $f$ yields
\[
    d_{\text{KR}}(\mu,\langle\mu\rangle_{m}) \le\frac{\mu(S^1 \times [0,T])}{m}.
\]
Since $d_{\text{KR}}$ metrizes the weak topology the claim that $\langle\mu\rangle_{m}$ weakly converges to $\mu$ follows.
\end{proof}

In the rest of the proof, without loss of generality, we take $T=1$. Also, for the proof to be more trackable 
we choose $\varepsilon =\varepsilon(n)$ so that $n c(\varepsilon)=1$, and denote $\mu_{n} = \mu_{\varepsilon(n)}$, where the driving $\mu_{\varepsilon}$ is defined in (\ref{eq:mu_delta}). The proof of the LDP is essentially the same for anisotropic HL$(0)$ model, so we denote by $\lambda$ the distribution on $S^{1}$ of attachment angles $(\theta_{k})$; in particular, set $\lambda\equiv 1/{2\pi}$ for the classical HL$(0)$ model.

Let us now consider the approximation of the driving measure $\mu_n$:
\begin{align*}
   \mu_{n,m}(d\theta,dt) = \langle \mu_{n}\rangle_{m}= \sum\limits_{i=1}^{m}\left(m \sum\limits_{k=1}^{n}\delta_{\theta_{k}}(d\theta) \scalebox{1}{$\left|\left[\frac{i-1}{m},\frac{i}{m}\right)\cap\left[\frac{k-1}{n},\frac{k}{n}\right)\right|$}\right)\mathds{1}_{\left[\frac{i-1}{m},\frac{i}{m}\right)}(t)dt,
\end{align*}
here $|A|$ denotes the Lebesgue measure of a set $A\subset \mathbb{R}$. This approximation can be decomposed as
\begin{align*}
   \mu_{n,m}(d\theta,dt) = \sum\limits_{i=1}^{m}\scalebox{1}{$\frac{m}{n}\left(\left\lfloor\frac{i}{m}n\right\rfloor-\left\lceil\frac{i-1}{m}n\right\rceil+1\right)$}L_{i}^{n,m}(d\theta)\mathds{1}_{\left[\frac{i-1}{m},\frac{i}{m}\right)}(t)dt + r_{n,m}(d\theta,dt),
\end{align*}
where
\begin{align*}
    L_{i}^{n,m}(d\theta) =\frac{1}{\scalebox{1}{$\left(\left\lfloor\frac{i}{m}n\right\rfloor-\left\lceil\frac{i-1}{m}n\right\rceil+1\right)$}}\sum\limits_{k=\left\lceil\frac{i-1}{m}n\right\rceil}^{\left\lfloor\frac{i}{m}n\right\rfloor}\delta_{\theta_{k}}(d\theta),
\end{align*}
and
\begin{align*}
    r_{n,m}(d\theta, dt) =  m\sum\limits_{i=1}^{m}\left(\delta_{\theta_{\left\lceil\frac{i-1}{m}n\right\rceil}}(d\theta){\scalebox{1}{$\left(\left\lceil\frac{i-1}{m}n\right\rceil\frac{1}{n} - \frac{i-1}{m}\right)$}}+\delta_{\theta_{\left\lceil\frac{i}{m}n\right\rceil}}(d\theta)\scalebox{1}{$\left(\frac{i}{m} -\left\lfloor\frac{i}{m}n\right\rfloor\frac{1}{n} \right)$}\right)\mathds{1}_{\left[\frac{i-1}{m},\frac{i}{m}\right)}(t)dt.
\end{align*}
The approximating measure $\mu_{n,m}$ is broken down into two parts. The first consists of $(L_{i}^{n,m})_{i=1}^{m}$, which are mutually independent empirical measures. The $L_{i}^{n,m}$ is made up of those $\theta_{k}$ that correspond to the $1/n$-intervals fully contained in $[\frac{i-1}{m},\frac{i}{m})$. The second part $r_{n,m}$ is the remainder, which is built from those $\theta_k$ that correspond to the $1/n$-intervals intersecting two neighboring $1/m$-intervals. See Figure \ref{figure:intervals} for an illustration.

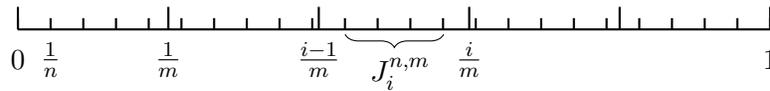
\begin{figure}[H]
\centering
\begin{tikzpicture}
    \def\len{10}
    \draw[thick] (0,0) -- (\len,0); 
    \def\n{23} 
    \def\m{5}  
    \foreach \i in {0,1,...,\n}
    {
        \draw[thick] (\len*\i/\n, 0.15) -- (\len*\i/\n, 0); 
    }
    \foreach \i in {0,1,...,\m}
    {
        \draw[thick] (\len*\i/\m, 0.3) -- (\len*\i/\m, 0); 
    }
    \draw [decorate,decoration={brace,amplitude=5pt,mirror,raise=1ex}]
    (4.35,0.1) -- (5.66,0.1) node[midway,yshift=-1.8em, xshift=0.2em]{$J_{i}^{n,m}$};
    \node at (0,-0.4) {$0$};
    \node at (\len/\n,-0.4) {$\frac{1}{n}$};
    \node at (\len/\m,-0.4) {$\frac{1}{m}$};
    \node at (2*\len/\m,-0.4) {$\frac{i-1}{m}$};
    \node at (3*\len/\m,-0.4) {$\frac{i}{m}$};
    \node at (\len,-0.4) {$1$};
\end{tikzpicture}
\captionsetup{width=.7\linewidth}
\caption{The interval $J_{i}^{n,m}= \left[ \left\lceil\frac{i-1}{m}n\right\rceil\frac{1}{n}, \left\lfloor\frac{i}{m}n\right\rfloor\frac{1}{n}\right)$ consists of those intervals of length $1/n$ that are fully contained in $[\frac{i-1}{m},\frac{i}{m})$.}
\label{figure:intervals}
\end{figure}

The constant factors \scalebox{0.9}{$\frac{m}{n}\left(\left\lfloor\frac{i}{m}n\right\rfloor-\left\lceil\frac{i-1}{m}n\right\rceil+1\right)$} and the remainder measure $r_{n,m}$ are irrelevant to the the large deviations of $\mu_{n,m}$, and the following measure can be considered instead:
\begin{align*}
    \tilde\mu_{n,m} =  \sum\limits_{i=1}^{m}L_{i}^{n,m}(d\theta)\mathds{1}_{\left[\frac{i-1}{m},\frac{i}{m}\right)}(t)dt
\end{align*}
This is due to the fact that the two measures, $\mu_{n,m}$ and $\tilde\mu_{n,m}$, are \textit{exponentially equivalent} as the next lemma demonstrates.
\begin{lemma}\label{lemma:exponential_equivalence}
For any $\delta>0$,  $\uplim\limits_{n\to\infty}\frac{1}{n}\log\mathbb{P}\left[d_{\text{KR}}(\mu_{n,m}, \tilde\mu_{n,m})>\delta\right] = -\infty.$
\end{lemma}
\begin{proof}
\begin{align*}
    ||\mu_{n,m} - \tilde\mu_{n,m}||_{TV}\le \frac{1}{m}\sum\limits_{i=1}^{m}\left|\scalebox{1}{$\frac{m}{n}\left(\left\lfloor\frac{i}{m}n\right\rfloor-\left\lceil\frac{i-1}{m}n\right\rceil+1\right)$} - 1\right| + \frac{m}{n}\to 0 \text{ as }n\to\infty.
\end{align*}   
From the inequality $||\cdot||_{KR}\le||\cdot||_{TV}$, we have $\mathbb{P}[d_{\text{KR}}(\mu_{n,m}, \tilde\mu_{n,m})>\delta] = 0$ for $n>N(m,\delta)$, which implies the claim.
\end{proof}
\noindent Exponentially equivalent random elements satisfy the same LDP, see \cite[Theorem 4.2.13]{dembo2009large}. Thus, we can concentrate our attention on $\tilde \mu_{n,m}$, which has the advantage that the summands $L_{i}^{n,m}$ are mutually independent. 

$\mathbf{1^{\circ}}.$ The large deviations of the empirical measure $L_{i}^{n,m}$ is the content of Sanov's theorem, see \cite[Theorem 3.2.17]{deuschel2001large} or \cite[Theorem 6.2.10]{dembo2009large}.
\begin{theorem}[Sanov's theorem]
The empirical measures $L_{i}^{n,m}$, indexed by $n$, satisfy a large deviation principle in $\mathcal{M}_{1}(S^{1})$, equipped with the weak topology, with rate $n/m$ and convex, good rate function $H(\cdot|\lambda)$, given by the relative entropy.
\end{theorem}
\noindent In other words, for any open set $O$ and any closed set $C$ of $\mathcal{M}_{1}(S^{1})$, 
\[
\text{(i)\quad } \limsup\limits_{n/m\to\infty}\frac{1}{n/m}\log\mathbb{P}\left[L_{i}^{n,m}(d\theta) \in C\right]\le-\inf\limits_{\nu\in C}H(\nu|\lambda)
\]
and 
\[
\text{(ii)\quad } \limsup\limits_{n/m\to\infty}\frac{1}{n/m}\log\mathbb{P}\left[L_{i}^{n,m}(d\theta) \in O\right] \ge  -\inf\limits_{\nu\in O}H(\nu|\lambda).
\]
\noindent Note that if the rate of convergence is considered to be $n$, then the rate function is simply scaled by $1/m$.

$\mathbf{2^{\circ}}.$ Every sequence, indexed by $n$, in the collection $L_1^{n,m}$, ..., $L_{m}^{n,m}$ satisfies the LDP in $\mathcal{M}_{1}(S^{1})$ with rate $n$ and common rate function $\frac{1}{m}H(\cdot|\lambda)$. The sequence $(L_{1}^{n,m}$, ..., $L_{m}^{n,m})$  as an element of $\mathcal{M}_{1}(S^{1})^{m}$ is made out of mutually independent empirical measures and thus satisfies the LDP with the rate function $I_m:\mathcal{M}_{1}(S^{1})^{m}\to[0,\infty]$ given by
\begin{align*}
    I_m(\nu_1,...,\nu_m) = \frac{1}{m}\sum\limits_{i=1}^{m}H(\nu_i|\lambda).
\end{align*} This general fact relies only on independence, see for example \cite[Corollary 1.3]{puhalskii2006large}.\\

$\mathbf{3^{\circ}}.$ The passage from the vector $(L_1^{n,m}$, ..., $L_{m}^{n,m})$ in $\mathcal{M}_{1}(S^{1})^m$ to the measure $\tilde \mu_{n,m}$ in $\mathcal{M}_{1}(S^1 \times [0,1])$ via the mapping  
$$(\nu_{1},...,\nu_{m}) \to \sum\limits_{i=1}^m\nu_{i} (d\theta)\mathds{1}_{\left[\frac{i-1}{m}, \frac{i}{m}\right)}(t)dt$$ is a direct application of the contraction principle, \cite[Theorem 4.2.1]{dembo2009large}. The continuity of the mapping is proved in the following lemma. 
\begin{lemma}\label{lemma:continuous_mapping_L}
    The mapping from $\mathcal{M}_{1}(S^{1})^{m}$ to $\mathcal{M}_{1}(S^1 \times [0,1])$ given by $$(\nu_1, ...,\nu_m)\to\sum\limits_{i=1}^{m}\nu_i(d\theta)\mathds{1}_{[\frac{i-1}{m},\frac{i}{m})}(t)dt$$ is continuous in the weak topology.
\end{lemma}
\begin{proof} 
The distance between the two measures $\mu(d\theta,dt) =\sum\limits_{i=1}^{m}\nu_i(d\theta)\mathds{1}_{[\frac{i-1}{m}, \frac{i}{n})}(t)dt$ and $\tilde\mu(d\theta,dt) = \sum\limits_{i=1}^{m}\tilde\nu_i(d\theta)\mathds{1}_{[\frac{i-1}{m}\frac{i}{n})}(t)dt$ is
    \begin{align*}
        &d_{\text{KR}}\left(\mu,\tilde\mu\right)=\sup\limits_{f}\left|\sum\limits_{i=1}^{m}\left(\int_{0}^{1}\int_{\frac{i-1}{m}}^{\frac{i}{m}}f(\theta,t)dt\nu_{i}(d\theta) - \int_{0}^{1}\int_{\frac{i-1}{m}}^{\frac{i}{m}}f(\theta,t)dt\tilde\nu_{i}(d\theta) \right)\right|,
    \end{align*}
    where the supremum is taken over Lip$_{1}(S^1 \times [0,1])$ functions with $\sup_{(\theta, t)\in S^1 \times [0,1]}|f(\theta, t)|\le 1$. The function $\theta \mapsto m\int_{\frac{i-1}{m}}^{\frac{i}{m}}f(\theta,t)dt$ is again Lipschitz continuous, and moreover the norm does not increase. Thus the distance can be bounded by 
    \begin{align*}
d_{\text{KR}}\left(\mu,\tilde\mu\right)
&\le \sum\limits_{i=1}^{m}  \sup\limits_{f}\left|\nu_{i}\left(m\int_{\frac{i-1}{m}}^{\frac{i}{m}}f(\cdot, t)dt\right) - \tilde\nu_{i}\left(m\int_{\frac{i-1}{m}}^{\frac{i}{m}}f(\cdot, t)dt\right)\right|\\
        &\le\frac{1}{m}\sum\limits_{i=1}^{m}\sup\limits_{g}\left|\nu_{i}(g)-\tilde\nu_{i}(g)\right|\\
        &=\frac{1}{m}\sum\limits_{i=1}^{m} d_{\text{KR}}(\nu_i,\tilde\nu_i).\qedhere
    \end{align*}
\end{proof}
\begin{lemma}\label{lemma_LDP_for_approximation}
    The measures $\tilde \mu_{n,m} =\sum\limits_{i=1}^mL_{i}^{n,m}(d\theta)\mathds{1}_{\left[\frac{i-1}{m}, \frac{i}{m}\right)}(t)dt$, indexed by $n$, satisfy a large deviation principle in $\mathcal{M}_{1}(S^1 \times [0,1])$ with the rate $n$ and the rate function 
\begin{align*}
    I_m(\mu) = 
    \begin{cases}
        \frac{1}{m}\sum\limits_{i=1}^{m}H(\nu_i|\lambda), \text{ if }\mu(d\theta,dt) = \sum\limits_{i=1}^{m}\nu_i(d\theta)\mathds{1}_{[\frac{i-1}{m},\frac{i}{m})}(t)dt \text{ with } \nu_i\in\mathcal{M}_{1}(S^{1}),\\
        +\infty, \text{ otherwise}.
    \end{cases}
\end{align*}
\end{lemma}
\begin{proof}
Let $\mathcal{L}:(\nu_{1},...,\nu_{m})\to\sum\limits_{i=1}^m\nu_{i}(d\theta)\mathds{1}_{\left[\frac{i-1}{m}, \frac{i}{m}\right)}(t)dt$ denote the continuous mapping of Lemma~\ref{lemma:continuous_mapping_L}.
By the contraction principle, see \cite[Theorem 4.2.1]{dembo2009large}, the rate function is given by 
\begin{align*}
    I_m(\mu) = \inf\{I_m(\nu_{1},...,\nu_{m}): (\nu_{1},...,\nu_{m}) \in \mathcal{L}^{-1}(\mu) \},\quad \mu\in\mathcal{M}_{1}(S^1 \times [0,1]).
\end{align*}
If $\mu\in\mathcal{M}_{1}(S^1 \times [0,1])$ has the form 
\[
\mu(d\theta,dt) = \sum\limits_{i=1}^m\nu_{i}(d\theta)\mathds{1}_{\left[\frac{i-1}{m}, \frac{i}{m}\right)}(t)dt,\quad \nu_i\in\mathcal{M}_{1}(S^{1}),
\]
then $\mathcal{L}^{-1}(\mu) = \{(\nu_{1},...,\nu_{m}) \}$, and 
\[
I_{m}(\mu) = \frac{1}{m}\sum\limits_{i=1}^{m}H(\nu_i|\lambda).
\]
Otherwise, $\mathcal{L}^{-1}(\mu) = \varnothing$, and the convention is that the infinum over an empty set is $+\infty$. 
\end{proof}
$\mathbf{4^{\circ}}.$ Lemma \ref{lemma:approximation} implies that $\mu_{n,m}$ is within distance $1/m$ to $\mu_{n}$ uniformly in $n$, since $\mu_{n}(S^1 \times [0,1]) = 1$. This, in turn, implies exponential closeness, that is, for any $\delta>0$
\begin{align*}
     \lim\limits_{m\to\infty}\uplim\limits_{n\to\infty}\frac{1}{n}\log\mathbb{P}\left[d_{\text{KR}}(\mu_{n},\mu_{n,m})>\delta\right] = -\infty,
\end{align*}
since $\mathbb{P}\left[d_{\text{KR}}(\mu_{n},\mu_{n,m})>\delta\right] =0$ for any $m>1/\delta$.

The combination of the exponential closeness of $\mu_{n}$ and $\mu_{n,m}$, the exponential equivalence of $\mu_{n,m}$ and $\tilde\mu_{n,m}$, Lemma \ref{lemma:exponential_equivalence}, and the LDP of $\tilde\mu_{n,m}$, Lemma \ref{lemma_LDP_for_approximation}, together with the standard result on exponential approximations \cite[Theorem 4.2.16]{dembo2009large} implies that the driving measures $\mu_{n}$ satisfy the weak LDP with rate $n$ and the rate function given by 
\begin{align*}
    I(\mu) = \sup\limits_{\delta>0}\lowlim\limits_{m\to\infty}I_{m}(B(\mu,\delta)).
\end{align*}
The next lemma shows that $I$ equals the relative entropy and as a consequence is a good rate function. First, recall the definition of the relative entropy: for any $\mu,\tilde \mu\in\mathcal{M}_{1}(S^1 \times [0,1])$
\begin{align*}
    H(\mu|\tilde\mu)  
    = 
    \begin{cases}
        \int_{S^1 \times [0,1]}\frac{d\mu}{d\tilde\mu}\log\frac{d\mu}{d\tilde\mu} d\tilde\mu, &\text{ if }\mu \ll\tilde\mu\\
        +\infty, &\text{ otherwise}.
    \end{cases}
\end{align*}
\begin{lemma}
  For any $\mu\in\mathcal{M}_{1}(S^1 \times [0,1])$, $I(\mu) = H(\mu(d\theta, dt)|\lambda(d\theta) dt)$, where $H$ is the relative entropy. Moreover, $I$ is a good rate function. 
\end{lemma}

\begin{proof}
Recall the rate function of $\mu_{n,m}$:
\begin{align*}
  &I_m(\mu) = 
    \begin{cases}
        \frac{1}{m}\sum\limits_{i=1}^{m}H(\nu_i|\lambda), &\text{ if }\mu(d\theta,dt) = \sum\limits_{i=1}^{m}\nu_i(d\theta)\mathds{1}_{[\frac{i-1}{m},\frac{i}{m})}(t)dt \text{ with } \nu_i\in\mathcal{M}_{1}(S^{1}),\\
        \infty, &\text{ otherwise}.
    \end{cases}
\end{align*}
From the approximate contraction principle the rate function of $\mu_{n}$ is given by
\begin{align*}
    I(\mu)=\sup\limits_{\delta>0}\lowlim\limits_{m\to\infty}I_m\left(B(\mu,\delta)\right).
\end{align*}

First, we prove that $I(\mu)\le H(\mu|\lambda(d\theta) dt)$. Assume $\mu(d\theta, dt) = \rho_{t}(\theta)\lambda(d\theta)dt$ and that $\iint \rho_{t}(\theta)\log\rho_{t}(\theta) \lambda(d\theta)dt<\infty$, for otherwise the inequality is trivially satisfied. By Lemma \ref{lemma:approximation}, the approximation $\langle\mu\rangle_m$ can be made arbitrary close to $\mu$ in the Kantorovich–Rubinshtein distance, i.e., for any $\delta>0$ there exists $m_{\delta}\in\mathbb{N}$ such that for all $m\ge m_{\delta}$ it holds that $d_{KR}(\mu,\langle\mu\rangle_{m})<\delta$.  Thus, the inclusion $\langle\mu\rangle_{m}\in B(\mu,\delta)$ yields $I_{m}(B(\mu,\delta)) \le I_{m}(\langle\mu\rangle_m)$. However, the bound  $I_{m}(\langle\mu\rangle_m)$ allows for the explicit formula
\begin{align*}
    I_{m}(\langle\mu\rangle_m) 
    = \frac{1}{m}\sum\limits_{i=1}^{m}H\left(m\int_{\frac{i-1}{m}}^{\frac{i}{m}}\rho_{s}(\theta)ds \lambda(d\theta)\Bigg| \lambda(d\theta)\right).
\end{align*}
The summands can be bounded with the help of Jensen's inequality as the function $f(x)=x\log x$ is convex:
\begin{align*}
    f\left(m\int_{\frac{i-1}{m}}^{\frac{i}{m}}\rho_{s}(\theta)ds\right) 
    \le m\int_{\frac{i-1}{m}}^{\frac{i}{m}}f\left(\rho_{s}(\theta)\right)ds.
\end{align*}
Integrating with respect to $\lambda(d\theta)$ gives
\begin{align*}
    H\left(m\int_{\frac{i-1}{m}}^{\frac{i}{m}}\rho_{s}(\theta)ds \lambda(d\theta)\Bigg| \lambda(d\theta)\right)
    \le \int_{0}^{2\pi} m\int_{\frac{i-1}{m}}^{\frac{i}{m}}f\left(\rho_{s}(\theta)\right)ds \lambda(d\theta),
\end{align*}
which translates to $I_{m}(\langle\mu\rangle_{m})\le H(\mu|\lambda(d\theta) dt)$. Thus, the bound 
\begin{align*}
    \lowlim\limits_{m\to\infty}I_{m}(B(\mu,\delta))\le H(\mu|\lambda(d\theta) dt),
\end{align*}
is obtained, and after taking taking the supremum in $\delta>0$,  we arrive at
\begin{align*}
    I(\mu)\le H(\mu|\lambda(d\theta) dt).
\end{align*}

Second,  we prove the reverse inequality $I(\mu)\ge H(\mu|\lambda(d\theta) dt)$. Lower semi-continuity of the relative entropy $H$ can be formulated as
\begin{align*}
    H(\mu|\lambda(d\theta) dt) = \lim\limits_{\delta\downarrow 0}\inf\limits_{\tilde\mu \in B(\mu,\delta)}H(\tilde\mu|\lambda(d\theta) dt).
\end{align*}
Combining it with the identity 
\begin{align*}
    H(\tilde\mu|\lambda(d\theta)dt) = I_{m}(\tilde\mu), \text{ whenever } \tilde\mu(d\theta,dt) = \sum\limits_{i=1}^{m}\nu_{i}(d\theta)\mathds{1}_{[\frac{i-1}{m},\frac{i}{m})}(t)dt,
\end{align*}
we obtain the following chain of inequalities 
\begin{align*}
    &\inf\limits_{\tilde\mu\in B(\mu,\delta)}H(\tilde\mu|\lambda(d\theta) dt)\\
    &\le \inf\limits \left\{H(\tilde\mu|\lambda(d\theta) dt): \tilde\mu \in B(\mu,\delta) \text{ and }\tilde\mu(d\theta,dt) = \sum\limits_{i=1}^{m}\nu_{i}(d\theta)\mathds{1}_{[\frac{i-1}{m},\frac{i}{m})}(t)dt \right\}\\
    &\le\inf\limits_{\tilde\mu\in B(\mu,\delta)} I_{m}(\tilde\mu).
\end{align*}
The last inequality uses the fact that $I_m(\tilde\mu)=\infty$ if $\tilde\mu$ is not of the form \scalebox{0.9}{$\sum\limits_{i=1}^{m}\nu_{i}(d\theta)\mathds{1}_{[\frac{i-1}{m},\frac{i}{m})}(t)dt$}, thus allowing the addition of such measures to set, over which the infinum is taken, without lowering the value of the infinum. After taking the limit in $m$ and then in $\delta$, we arrive at the desired bound
\begin{align*}
    H(\mu|\lambda(d\theta) dt) \le \sup\limits_{\delta>0} \lowlim\limits_{m\to\infty}\inf\limits_{\tilde\mu\in B(\mu,\delta)}I_m(\tilde\mu) = I(\mu). 
\end{align*} 

Therefore, combining the two bounds together,  the equality $I(\mu)=H(\mu|\lambda(d\theta) dt)$ is obtained. Moreover, since the relative entropy is a convex good rate function, so is $I$.\\
\end{proof}
\begin{remark}
    In the proof above we have made use of the inequality $I_{m}(\langle\mu\rangle_{m})\le H(\mu|\lambda(d\theta) dt))$, which is direct consequence of the Jensen inequality. In fact, one can show the convergence $\lim\limits_{m\to\infty}I_{m}(\langle\mu\rangle_{m})= H(\mu|\lambda(d\theta)dt)$, see \cite{kozlov_fine-grained_2007}, where the latter and the former quantities are called coarse-grained and fine-grained entropies correspondingly.
\end{remark}

By \cite[Theorem 8.9.3]{bogachev2007measure}, $\mathcal{M}_{1}(S^1 \times [0,1])$ is a compact space. Thus, one can conclude that the family of probability laws of $\mu_{n}$ is exponentially tight. The weak LDP of $\mu_{n}$ together with exponential tightness imply the full LDP of $\mu_n$ with a good rate function $I=H$, see \cite[Lemma 2.1.5]{deuschel2001large}. This concludes the proof of Theorem \ref{theorem_LDP}. \qed
\subsection{Proof sketch for Proposition \ref{P:LDPcontinuous}}
We will not provide a full proof of Proposition \ref{P:LDPcontinuous}, the large deviations result in continuous time, but we will content ourselves with sketching its proof, as it shares many ideas with the proof of Theorem \ref{theorem_LDP}. 

\begin{proof}[Sketch of proof of Proposition \ref{P:LDPcontinuous}] Since the Poisson process $(N(t), t\ge 0)$ is independent of the discrete time dynamics \eqref{Phi_n_mapping}, and given Theorem \ref{theorem_LDP}, it suffices to consider the large deviations of the total mass, i.e., we wish to obtain the asymptotics as $\delta \to 0$ of 
$$
\log \mathbb{P} ( \lambda^{-1} N (t) \in E)
$$
for any Borel measurable set of functions $\tilde m: [0, T] \to [0, \infty)$ equipped with the topology of uniform convergence, and where we recall that $\lambda = 4/\delta^2$ is the rate of the Poisson process (but is also the rate of the LDP in Theorem \ref{theorem_LDP}). 

Now we recall that if $Y_1, \ldots, Y_n$ are i.i.d. Poisson random variable with rate 1 then $\tfrac1n \sum_{i=1}^n Y_i$ satisfies an LDP on $\mathbb{R}$ with rate $n$ and rate function
$$
I(x) = x \log x - (x-1).
$$
Using the fact that arrivals of particles during $[t , t+ \eps]$ are given by a Poisson number of particles with mean $\lambda \eps$ (or, alternatively, a sum of $n = \lambda \eps$ i.i.d. Poisson random variables of rate 1), we deduce that $\lambda^{-1} N(t) $ satisfies an LDP with rate function 
$$
\int_{0}^T \tilde m_t \log \tilde m_t dt - \int_{0}^T (\tilde m_t -1)dt
$$
and rate $\lambda$. However, on the space $\tilde \cN_T$ we have necessarily $\int_0^\infty \tilde m_t dt  = T$ so that the rate function is just $\int_{0}^T \tilde m_t \log \tilde m_t dt = h_{\mathbb{R}} (\tilde{m})$, as desired.
\end{proof}

\section{Geometric properties of finite entropy Loewner chains}
\label{section_finite_entropy}
This section completes the proof of Theorem~\ref{prop:shapes}.

\subsection{Loewner--Kufarev energy and Weil-Petersson quasicircles}
\label{SS:energyinequality}


We now prove the entropy-energy inequality of Theorem~\ref{T:energy_entropy}.

\begin{proof}[Proof of Theorem~\ref{T:energy_entropy}]
We claim that if $\rho$ is a probability density on $S^{1}$ such that $\sqrt{\rho}$ is absolutely continuous and $\int_{0}^{2\pi}\rho'(\theta)^2/\rho(\theta)d\theta<\infty$ then the following inequality is satisfied:
\begin{align*}
   h_{S^1}(\rho) \le 2 \mathcal{D}(\sqrt{\rho}),
\end{align*}
where $h_{S^1}(\rho)$ is the relative entropy of $\rho d\theta$ with respect to $d\theta/2\pi$. This immediately implies the first claim of Theorem~\ref{T:energy_entropy}.
To prove the estimate, we use the following logarithmic Sobolev inequality 
\begin{align*}
    \int_{0}^{2\pi}f^2\log f^2\frac{d\theta}{2\pi} -\int_{0}^{2\pi}f^2\frac{d\theta}{2\pi}\log\int_{0}^{2\pi}f^2\frac{d\theta}{2\pi} \le 2 \int_{0}^{2\pi}|f'|^2\frac{d\theta}{2\pi}, 
\end{align*}
which holds for all real $f$ such that $f'\in L^2(S^{1})$, $\int_{S^{1}}f'd\theta = 0$, and $f$ is continuous and is determined by $f'$ up to a constant; see \cite{emery_yukich_1987ASP}. Then, for $f(\theta) = \sqrt{2\pi \rho(\theta)}$ we obtain the desired inequality:
\begin{align*}
   h_{S^1}(\rho)= \int_{0}^{2\pi} \log\left(2\pi \rho(\theta)\right)\rho(\theta)d\theta \le \frac{1}{2} \int_{0}^{2\pi} \frac{\rho'(\theta)^2}{\rho(\theta)}d\theta= 2 \mathcal{D}(\sqrt{\rho}).
\end{align*}

The last statement of Theorem \ref{T:energy_entropy}, i.e., if $K$ is any compact hull bounded by a Weil--Petersson quasicircle, then $K$ can be generated by a measure with finite Loewner--Kufarev entropy, now follows from the inequality and Theorem~1.4 of \cite{viklund_loewnerkufarev_2024}.
\end{proof}
\begin{remark}Another link between the entropy $\mathcal{H}$ and energy $S$ comes from the de Bruijn identity. Let $\rho_0$ be a probability density on $S^{1}$ with $ \rho_0'(0) = \rho_0'(2\pi) = 0$. Let $\rho_t$ solve the heat equation for $t \ge 0$, so that $\dot{\rho_t} = \rho_t''/2$, with Neumann boundary condition  $\rho_t'(0) =  \rho_t'(2\pi) = 0$. Then $\int_0^{2\pi}\rho_t dx/2\pi \equiv 1$ so we obtain a measure $\rho = (\rho_t)_{\rho \ge 0}$ whose marginals interpolate between $\rho_0$ and the stationary solution $\rho_\infty \equiv 2\pi$. Using the heat equation and integration by parts,
\begin{align*}
    \frac{d}{dt}\int_0^{2\pi} \rho_t \log \rho_t d\theta & = \int_0^{2\pi} \dot{\rho_t}(\log \rho_t + 1) d\theta \\
    & = \frac{1}{2} \int_0^{2\pi} \rho''_t \log \rho_t d\theta \\
    & = -\mathcal{D}(\sqrt{\rho_t}) \le 0,
\end{align*}
with strict inequality unless $\rho_0 \equiv 2\pi$. In this setting the Dirichlet integral on the right is sometimes called the Fisher information.
Hence, for $T \ge 0$, for this particular choice of measure,
\[
\int_0^{2\pi} \rho_T \log \rho_T d\theta -  \int_0^{2\pi} \rho_0 \log \rho_0 d\theta = - S_+(\rho\mid_{[0,T]}).
\]
Since $\rho_t$ converges to the stationary solution with vanishing entropy as $t \to \infty$, by dominated convergence we obtain interesting formula
\[
S_+(\rho) = \int_0^{2\pi} \rho_0 \log \rho_0 d\theta.
\]
\end{remark}
\subsection{Becker quasicircles}\label{sect:becker}
Becker \cite{Becker1972} gave a sufficient condition on the driving measure $\rho$ to generate quasicircles via the Loewner equation, see also \cite{GumenyukPrause2018}. Consider the Loewner--Kufarev equation \eqref{eq:LK-eq}. If $w$ is on the unit circle, the Möbius map $z \mapsto \tfrac{z+w}{z-w}$ maps the exterior disc $\Delta = \hat{\mathbb{C}} \smallsetminus \mathbb{D}$ to the half plane of points with positive real parts. Therefore, the function $H_t$ appearing in \eqref{eq:LK-eq} 
is analytic and has positive real part in $\Delta$ (such a function is sometimes called a Herglotz function). One version of Becker's condition can be stated as follows. If there exists $\kappa \in [0,1)$ such that for a.e., $t \in [0,T]$,
\begin{equation}\label{condition:becker}
H_t(\Delta) \subset \left\{w\in \mathbb{C}: \left|\frac{1-w}{1+w} \right| \le \kappa \right\},
\end{equation}
then $f_t(S^1)$ is a quasicircle for every $t \in [0,T]$, where $(f_t)_{t \in [0,T]}$ is the solution to \eqref{eq:LK-eq}. Note that $w \mapsto (1-w)/(1+w)$ is a M\"obius transformation of the right half-plane onto the unit disc with $1$ mapped to $0$. Since $\kappa <1$, \eqref{condition:becker} implies that $H_t(\Delta)$ is contained in a compact subset of the right half-plane containing $1$. We say that $\gamma$ is a (finite time) Becker quasicircle if $\gamma = f_t(S^1)$ for some Loewner chain satisfying \eqref{condition:becker}. Neither the Becker condition nor the $S(\rho) < \infty$ Weil-Petersson condition implies the other. Note that $\kappa = 0$ if and only if $\rho_t = d\theta dt/2\pi$. 
\begin{proposition}\label{prop:becker}
    If $\rho \in \mathcal{N}_T$ satisfies \eqref{condition:becker}, then $\mathcal{H}(\rho) < CT$, where $C < \infty$ depends only on $\kappa$.
\end{proposition}

\begin{proof}
    Suppose $H_t$ satisfies \eqref{condition:becker}. Then $H_t$ is a bounded analytic function in $\Delta$. It follows (see e.g. Corollary 2 to Theorem~3.1 of \cite{DurenHp}) that $\Re \, H_t$ is the Poisson integral of an $L^\infty(S^1)$ function $\rho_t$, that is, the measure $\rho_t(d\theta)$ is absolutely continuous with $L^\infty$ density $\rho_t$. Moreover, $M:=\|\rho_t\|_\infty$ depends only on $\kappa$. Next we note that there exists $C$ depending only on $\kappa$ such that,
    \[
   \frac{1}{2\pi}\int_{S^1}\rho_t(\theta) \log \rho_t(\theta) d\theta  \le \sup_{r \in [0, M]}\left|r \log r \right| < C.
    \]
  Integrating over $[0,T]$, we see that $\mathcal{H}(\rho)< C T < \infty$, as claimed. (In fact, a similar argument can be used to show that if $\Re H_t (\Delta)$ is bounded, then the same conclusion applies.)  
\end{proof}
\subsection{Examples of finite entropy shapes}

\label{SS:examples}

We now discuss several examples where the entropy can be computed or estimated. These examples combined with Theorem~\ref{T:energy_entropy} prove Theorem~\ref{prop:shapes}.
\subsubsection*{Inward-pointing corner}
This example is based on \cite{theodosiadis}. 
The conformal map $\varphi_{t}(z)=e^{t}\frac{z-1}{2}+\frac{1}{2}\sqrt{4z+e^{2t}(1-z)^2}
$ takes the unit disk $\mathbb{D}$ to the union $\mathbb{D}\cup B(e^t, \sqrt{e^{2t}-1})$. Rotating the hull and applying the Schwarz reflection we get the map 
$f_{t}(z) = -1/\varphi_{t}(-1/z)$ that takes the exterior disk $\Delta$ to the complement of the compact hull. Figure \ref{fig:orthogonal_disks} illustrates the hull for different values of $t>0$.
\begin{figure}[H]
    \centering
    \includegraphics[width=0.5\linewidth]{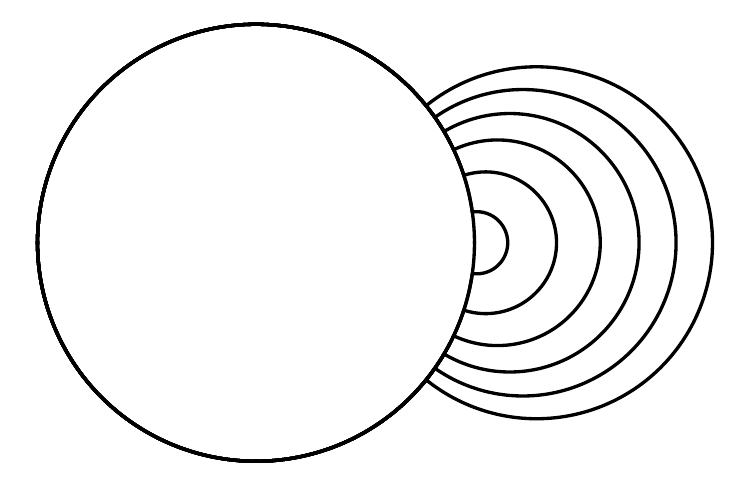}
    \caption{The growing hull given by the map $f_t: \Delta \to \Delta\smallsetminus B(e^t, \sqrt{e^{2t}-1})$ }
    \label{fig:orthogonal_disks}
\end{figure}
\noindent The probability density that generates such hulls is
\[
    \rho_{t}(\theta) 
    = \tilde\rho_{t}(\pi+\theta)\mathds{1}_{[0,\pi)}(\theta) 
    + \tilde\rho(\pi-\theta)\mathds{1}_{[\pi,2\pi]}(\theta),
\]
where
$$\tilde\rho_{t}(\theta) = \left(\frac{e^{-t}}{\pi(1-e^{-2t})}\sin\left(\frac{\theta}{2}\right)\sqrt{e^{2t}\sin^2\left(\frac{\theta}{2}\right)-1}\right)\mathds{1}_{I_{t}}(\theta),$$
and $I_{t} = [2 \arcsin(e^{-t}), 2\pi - 2\arcsin(e^{-t})].$

For a fixed time $t>0$, the relative entropy of this density with respect to the uniform measure on $S^{1}$ equals
\begin{align*}h_{S^{1}}(\rho_{t}) &=\int_{0}^{2\pi}\tilde\rho_{t}(\theta)\log(2\pi\tilde\rho_{t}(\theta)) d\theta\\
&= 
\log\left(\frac{2}{e^{2t}-1}\right) + \frac{2}{\pi(e^{2t}-1)}\int_{1}^{e^{t}}x\sqrt{x^2-1}\log\left(x\sqrt{x^2-1}\right)\frac{dx}{\sqrt{e^{2t}-x^2}}<+\infty.
\end{align*}
As $t\to0$ the probability density $\rho_{t}$ localizes at one point and the entropy $h_{S^{1}}(\rho_{t})$ diverges. However, as it diverges slower than $1/\sqrt{t}$, the total entropy $\mathcal{H}(\rho) = \int_{0}^{1}h_{S^{1}}(\rho_{t})dt$ is finite.
\begin{align*}
    \mathcal{H}(\rho) \le \int_{0}^{1}\underbrace{\log\left(\frac{2}{e^{2t}-1}\right)}_{\sim\log\frac{1}{t}}dt + \frac{2 e\sqrt{e^2-1}\log(e\sqrt{e^2-1})}{\pi}\int_{0}^{1}\underbrace{\frac{\arccos(e^{-t})}{e^{2t}-1}}_{\sim\frac{1}{\sqrt{t}}}dt <+\infty.
\end{align*}
\subsubsection*{Inward-pointing cusp and non-Jordan curve}
This example is also based on \cite{theodosiadis}. 
Consider the conformal maps
\begin{align*}
    \tilde f_{s}(z) = 1+ \frac{2\pi is }{\log\left(\frac{z-e^{-i\pi s}}{z-e^{i\pi s}}\right)}, \quad 0\le s<1,
\end{align*}
that map the exterior disk $\Delta = \{|z|>1\}$ to the complement of hulls depicted in Figure \ref{fig:tangent_disks}.
\begin{figure}[H]
    \centering
    \includegraphics[width=0.4\linewidth]{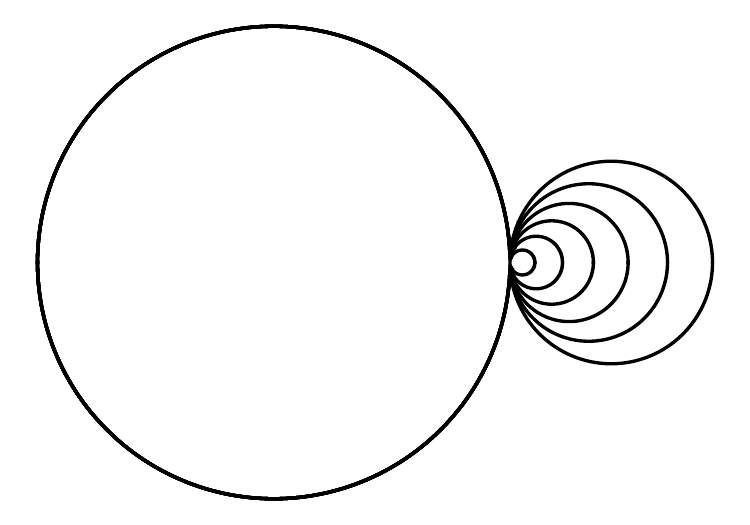}
    \caption{The hulls grown by the conformal map $\tilde f_{s}:\Delta \to \Delta\smallsetminus B(\frac{1}{1-s}, \frac{1+s}{1-s})$. Note that the interfaces are not Jordan curves.}
    \label{fig:tangent_disks}
\end{figure}
Note that the time-parametrization differs from the standard one. We have $\tilde f_s(z) = \beta(s)z + o(z)$, where 
$\beta(s) = \frac{\pi s}{\sin(\pi s)}$. One can go back to the standard capacity parametrization $f_{t}(z) = \tilde f_{\alpha(t)}(z) = e^{t} z + o(z)$ with the function $\alpha(t) =\beta^{-1}(e^{t})$.
The probability density is
\begin{align*}
    \tilde\rho_{s}(\theta) = \frac{\cos(\theta) - \cos(\pi s)}{2s\sin(\pi s)}\mathds{1}_{(-\pi s, \pi s )}(\theta), \quad 0\le s<1, \quad \int_{0}^{2\pi}\tilde\rho_{s}(\theta)d\theta =  \frac{\beta'(s)}{\beta(s)};
\end{align*}
or after time reparametrization one obtains 
\begin{align*}
    \rho_{t}(\theta) =\alpha'(t)\tilde 
    \rho_{\alpha(t)}(\theta) = \alpha'(t)\frac{\cos(\theta) - \cos(\pi \alpha(t))}{2\alpha(t)\sin(\pi \alpha(t))}\mathds{1}_{(-\pi \alpha(t), \pi \alpha(t) )}(\theta), \quad t>0.
\end{align*}

For a fixed time $t>0$, the relative entropy of the density with respect to the uniform measure on $S^{1}$ equals
$$h_{S^{1}}(\rho_t) = \log\left(\frac{\pi \alpha'(t)}{\alpha(t)\sin(\pi\alpha(t))}\right)+ \frac{\alpha'(t)}{\alpha\sin(\pi\alpha(t))}I(\alpha(t))<\infty,$$
where
\begin{align*}I(\alpha) 
&= \int_{0}^{\pi\alpha}(\cos(\theta)-\cos(\pi\alpha))\log(\cos(\theta)-\cos(\pi\alpha))d\theta\\
& = \int_{0}^{1-\cos(\pi \alpha)}\frac{x\log x}{\sqrt{1-(x+\cos(\pi\alpha))^2}}dx.
\end{align*}
The function $I$ can be bounded by 
\begin{align*}
    |I(\alpha)|
    \le \frac{1}{\sqrt{1+\cos(\pi\alpha)}}\left|\int_{0}^{1-\cos(\pi\alpha)}\frac{x\log x}{\sqrt{(1-\cos(\pi\alpha))-x}}dx\right|,
\end{align*}
where the integral is computed explicitly; for $a>0$,
\begin{align*}
    \int_{0}^{a}\frac{x\log x}{\sqrt{a-x}} dx
    = \frac{4}{9} a^{3/2}(3\log a-5+\log 64).
\end{align*}
Thus, as $\alpha\to 0$, $|I(\alpha)|\sim \alpha^{3}\log\frac{1}{\alpha}$.
The derivative of $\alpha$ is given by
$$\alpha'(t) = \frac{e^{t}\sin^2(\pi\alpha)}{\pi(\sin(\pi\alpha) - \pi\alpha\cos(\pi\alpha))},$$
so that
$$\frac{\alpha'}{\alpha\sin(\pi\alpha)} = \frac{e^{t}\frac{\sin(\pi\alpha)}{\pi\alpha}}{\sin(\pi\alpha)-\pi\alpha\cos(\pi\alpha)} =\frac{1}{\sin(\pi\alpha)-\pi\alpha\cos(\pi\alpha)}.$$
As $\alpha\to 0$ 
\begin{align*}
    \sin(\pi\alpha)-\pi\alpha\cos(\pi\alpha)\sim \frac{(\pi\alpha)^3}{3}+o(\alpha^3).
\end{align*}
Hence, it follows that the total entropy $\mathcal{H}(\rho)=\int_{0}^{1} h_{S^{1}}(\rho_{t})dt$ is finite
\begin{align*}
    \mathcal{H}(\rho) =\int_{0}^{1}\underbrace{\log\left(\frac{\pi}{\sin(\pi\alpha(t))-\pi\alpha(t)\cos(\pi\alpha(t))}\right)}_{\sim\log t}dt+\int_{0}^{\alpha(1)}\underbrace{\frac{I(\alpha)}{\alpha\sin(\pi\alpha)}}_{\sim \alpha\log\alpha}d\alpha<+\infty.
\end{align*}

\subsubsection*{Outward-pointing cusps}
Consider the time-independent probability density given by the Poisson kernel:
\begin{align*}
    \rho_{t}(\theta) \equiv  \frac{1}{2\pi}\frac{R^2-1}{1-2R\cos(\theta)+R^2},
\end{align*}
where $R>1$ is a fixed constant.
The conformal map that solves Loewner--Kufarev equation, driven by this density, is given by 
\begin{align*}
    f_t(z) = \frac{1}{R}\left[e^t\frac{Rz+1}{2Rz}\left(Rz+1+\sqrt{R^2z^2+1+2R z(1-2e^{-t})}\right)-1\right], \quad z\in\Delta.
\end{align*}
The derivation of this conformal map can be found, for example, in \cite{sola2013elementary} along with many other examples of Loewner chains driven by densities. Figure~\ref{fig:Poisson_constant_R} illustrates the hulls generated by the Poisson kernel density for different values of the parameter $R>1$.
\begin{figure}[h!]
     \centering
     \begin{subfigure}[b]{0.3\textwidth}
         \centering
         \includegraphics[width=\textwidth]
         {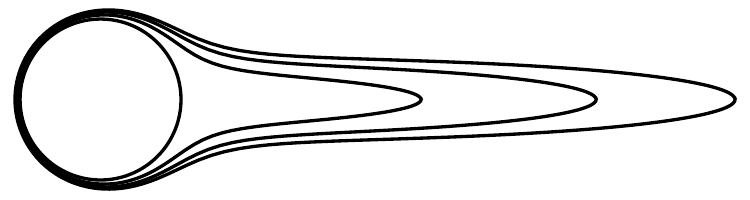}
         \vspace{+10pt}
         \caption{$R=1.1$}
     \end{subfigure}
     \hfill
     \begin{subfigure}[b]{0.3\textwidth}
         \centering
         \includegraphics[width=\textwidth]{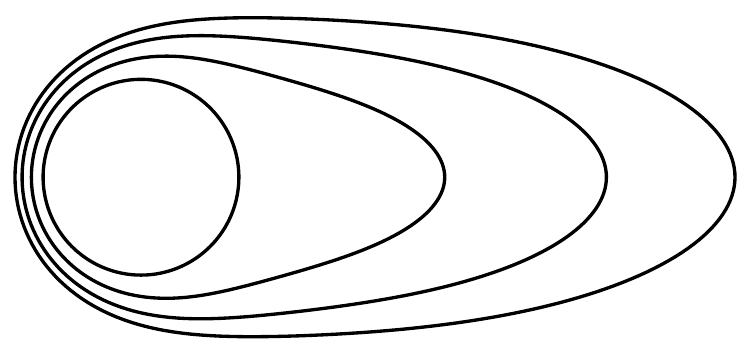}
         \vspace{+1pt}
         \caption{$R=1.5$}
     \end{subfigure}
     \hfill
     \begin{subfigure}[b]{0.3\textwidth}
         \centering
         \includegraphics[width=\textwidth]{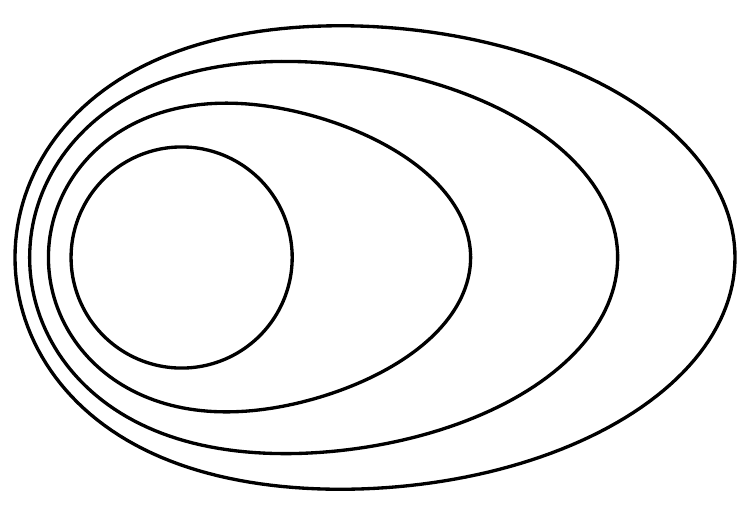}
         \caption{$R=2$}
     \end{subfigure}
        \caption{The hulls generated by the Poisson kernel density with constant parameter $R>1$.}
        \label{fig:Poisson_constant_R}
\end{figure}

For a fixed-time $t>0$, the relative entropy $h_{S^{1}}(\rho_{t})$ equals
\begin{align*}
    h_{S^{1}}(\rho_t) 
    &=\int_{0}^{2\pi}\frac{R^2-1}{1-2R\cos(\theta)+R^2}\log\left(\frac{R^2-1}{1-2R\cos(\theta)+R^2}\right)\frac{d\theta}{2\pi}\\
    &=\frac{1}{\pi} \int_{\frac{R-1}{R+1}}^{\frac{R+1}{R-1}}\frac{\frac{1}{x}\log\frac{1}{x}}{\sqrt{\left(\frac{R+1}{R-1} - x\right)\left(x-\frac{R-1}{R+1}\right)}}dx.
\end{align*}
After the change of variable $y= (x-\frac{R-1} {R+1})\frac{R^2-1}{4R}$ we obtain
\begin{align*}
    h_{S^{1}}(\rho_t) 
    &=\frac{1}{\pi}\int_{0}^{1}\frac{\frac{1}{\frac{R-1}{R+1}+ \frac{4R}{R^2-1}y}\log\frac{1}{\frac{R-1}{R+1}+ \frac{4R}{R^2-1}y}}{\sqrt{(1-y)y}}dy.
\end{align*}
Since $\frac{R-1}{R+1}\le \frac{R-1}{R+1}+ \frac{4R}{R^2-1}y \le \frac{R+1}{R-1}$, the logarithm takes both positive and negative values, and changes the sign at $y = \frac{R-1}{2R}$. We can write down the following estimate
\begin{align*}
h_{S^{1}}(\rho_t) &\le\frac{1}{\pi}\int_{0}^{\frac{R-1}{2R}}\frac{\frac{1}{\frac{R-1}{R+1}+ \frac{4R}{R^2-1}y}\log\frac{1}{\frac{R-1}{R+1}+ \frac{4R}{R^2-1}y}}{\sqrt{(1-y)y}}dy\\
    &\le \frac{1}{\pi}\frac{1}{\frac{R-1}{R+1}}\log\frac{1}{\frac{R-1}{R+1}}\int_{0}^{\frac{R-1}{2R}}\frac{1}{\sqrt{(1-y)y}}dy\\
    &\le \frac{2}{\pi}\arcsin\left(\sqrt{\frac{R-1}{2R}}\right)\frac{R+1}{R-1}\log\frac{R+1}{R-1}.
\end{align*}
    As $R\downarrow 1$, the upper bound essentially behaves like $\frac{1}{\sqrt{R-1}}\log\frac{1}{R-1}$.

Now, let us consider the probability density given by the Poisson kernel where the parameter $R=R(t)$ is time-dependent:
\begin{align*}
    \rho_{t}(\theta) = \frac{1}{2\pi}\frac{R(t)^2-1}{1-2R(t)\cos(\theta)+R(t)^2}.
\end{align*}
Three different choices of $R=R(t)$ are illustrated in Figure \ref{fig:time_dependent_R}.
\begin{figure}[h]
\begin{subfigure}{0.3\textwidth}\centering
\includegraphics[width=0.9\linewidth]{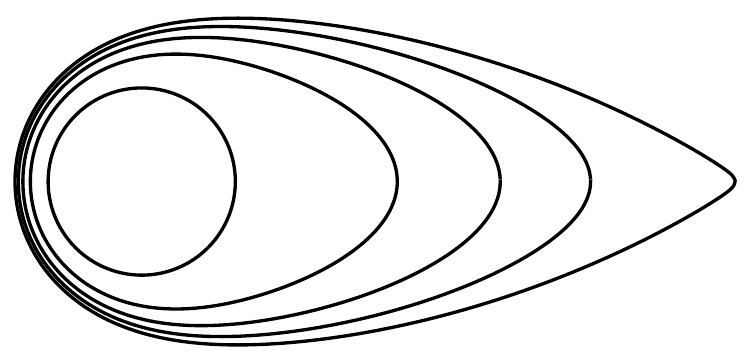} 
\caption{$R(t) = 1+\sqrt{1-t}$}
\label{fig:subim1}
\end{subfigure}
\begin{subfigure}{0.3\textwidth}\centering
\includegraphics[width=0.9\linewidth]{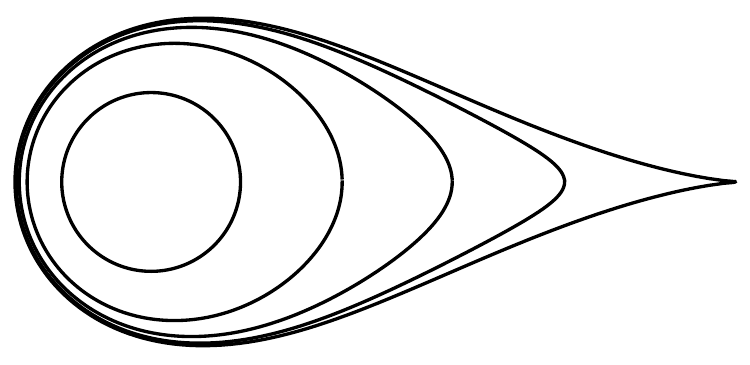}
\caption{$R(t) = \frac{1}{t}$}
\label{fig:subim2}
\end{subfigure}
\begin{subfigure}{0.3\textwidth}\centering
\includegraphics[width=0.9\linewidth]{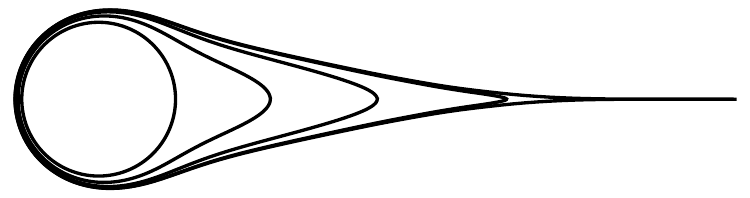}
\vspace{+10pt}
\caption{$R(t) = 1+e^{-\frac{1}{1-t}
}$}
\label{fig:subim3}
\end{subfigure}
\caption{The hulls grown by the Poisson kernel density with different choices of $R=R(t)$.}
\label{fig:time_dependent_R}
\end{figure}
The idea is that if we chose  $R=R(t)$ appropriately, so that $\lim\limits_{t\to 1}R(t)=1$, then at time $t=1$ the hull will form a cusp.
 
For a fixed $t>0$, the relative entropy $h_{S^{1}}(\rho_t)$ equals the same expression that we have computed for a constant parameter $R>1$ above:
\begin{align*}
    h_{S^{1}}(\rho_t) = \int_{0}^{2\pi}\rho_{t}(\theta)\log\rho_{t}(\theta)\frac{d\theta}{2\pi}\le \frac{2}{\pi}\arcsin\left(\sqrt{\frac{R(t)-1}{2R(t)}}\right)\frac{R(t)+1}{R(t)-1}\log\frac{R(t)+1}{R(t)-1}.
\end{align*}
The total entropy $\mathcal{H}(\rho)=\int_{0}^{1}h_{S^{1}}(\rho_{t})dt$ is bounded by
\begin{align*}
    \mathcal{H}(\rho) 
    \le  \int_{0}^{1} \frac{2}{\pi}\arcsin\left(\sqrt{\frac{R(t)-1}{2R(t)}}\right)\frac{R(t)+1}{R(t)-1}\log\frac{R(t)+1}{R(t)-1} dt.
\end{align*}
For both $R(t) = 1+ \sqrt{1-t}$ and $R(t)=1/t$, the upper bound if finite, so they provide examples of finite entropy hulls. 

Next step is to investigate the derivative of the conformal map $f_{t}$ at time $t=1$ around the point $z=1$. The conformal map $g_{t} = f_{t}^{-1}$, from $\Delta_{t} =f_{t}(\Delta)$ to $\Delta$, satisfies the Loewner equation 
\[
    \dot g_{t}(z)= - g_{t}(z)p_{t}(g_{t}(z)), \quad z\in \Delta_{t},
\]
with $g_{0}(z)=z$, and where
\[
    p_{t}(z)  =\int_{0}^{2\pi}\frac{z + e^{i\theta}}{z- e^{i\theta}}\rho_{t}(d\theta)
\]
is an analytic function with positive real part and $p_{t}(\infty)=1$. Differentiating the Loewner equation with respect to $z$ and then integrating with respect to $t$ one gets
\begin{align*}
    g_{t}'(z) = \frac{g_{t}(z)}{z}\exp\left(-\int_{0}^{t}g_{s}p_{s}'(g_{s})ds\right).
\end{align*}
This translates to the derivative of $f_{t}$ via  $f_{t}'(z) = 1/(g_{t}'(f_{t}(z)))$:
\begin{align*}
    f_{t}'(z) = \frac{f_{t}(z)}{z}\exp\left(\int_{0}^{t}g_{s}(f_{t}) p_{s}'(g_{s}(f_{t}))ds\right).
\end{align*}
With the help of the identity
\[
    \log\frac{p_{t}(g_{t}(z))}{p_{0}(z)} = \int_{0}^{t}\frac{\dot p_{s}(g_{s})+p_{s}'(s_{s})\dot g_{s}}{p_{s}(g_{s})}ds,
\]
which combined with the Loewner equation becomes
\begin{align*}
    \int_{0}^{t}g_{s}p_{s}'(g_{s})ds = \log\frac{p_{0}(z)}{p_{t}(g_{t})} + \int_{0}^{t}\frac{\dot p_{s}(g_{s})}{p_{s}(g_{s})}ds,
\end{align*}
the expression $f_{t}'(z)$ transforms into 
\begin{align*}
    f_{t}'(z) = \frac{f_{t}(z)p_{0}(f_{t}(z))}{zp_{t}(z)}\exp\left(\int_{0}^{t}\frac{\dot p_{s}(g_{s}\circ f_{t})}{p_{s}(g_{s}\circ f_{t})}ds\right).
\end{align*}

Specifying $\rho_{t}$ to be the Poisson kernel density one can compute $p_{t}(z)$ explicitly, see \cite{sola2013elementary} for details:  
\[
    p_{t}(z) = \frac{R_{t}z+1}{R_{t}z-1}.    
\]
Then, $f_t'$ becomes
\begin{align*}
    f_{t}'(z) = (R_{t}z-1)\frac{f_{t}(z)(R_{0}f_{t}+1)}{z(R_{0}f_{t}-1)(R_{t}z+1)}\exp\left(\int_{0}^{t}\frac{-2\dot R_{s}(g_{s}\circ f_t)(z)}{R^2_{s}(g_{s}\circ f_t)^2(z)-1}ds\right),
\end{align*}
and at $t=1$, that is, when $R_1=1$:
\begin{align*}
    f_{1}'(z) = (z-1)\frac{f_{1}(z)(R_{0}f_{1}(z)+1)}{z(R_{0}f_{1}(z)-1)(z+1)}\exp\left(\int_{0}^{1}\frac{-2(g_{s}\circ f_1)(z)}{R^2_{s}(g_{s}\circ f_1)^2(z)-1}\dot R_{s}ds\right).
\end{align*}
To control the function inside the exponent denote $\phi_{s} = g_{s}\circ f_{1}$ and write
\begin{align*}
    \int_{0}^{1}\frac{-2\phi_{s}(z)}{R^2_{s}\phi_{s}(z)^2-1}\dot R_{s}ds 
    = \int_{0}^{1}\frac{2\frac{\dot R_{s}}{R_{s}^2}}{(\phi_{s}(z)+\frac{1}{R_{s}})^2}\dot\phi_{s}(z)ds,
\end{align*}
where we used the identity 
\[
    \frac{-\phi_{s}(z)}{R_{s}\phi_{s}(z)-1} = \frac{\dot\phi_{s}(z)}{\phi_{s}(z)+1}
\]
which follows from the Loewner equation.

For the case $R_t = 1/t$, the integral inside the exponent in the expression for $f_{1}'$ is well-behaved:
\begin{align*}
    \int_{0}^{1}\frac{-2}{(\phi_{s}(z)+s)^2}\dot\phi_{s}(z)ds = \frac{2}{z+1} - 
    \frac{2}{f_{t}(z)}+ \int_{0}^{1}\frac{2}{(\phi_{s}(z)+s)^2} ds,
\end{align*}
where 
\begin{align*}
    \left|\int_{0}^{1}\frac{2}{(\phi_{s}(z)+s)^2} ds\right| \le \int_{0}^{1}\frac{2}{|z+s|^2} ds = \frac{2}{\Im (z)}\left(\text{arctan}\left(\frac{\Re(z)+1}{\Im(z)}\right)-\text{arctan}\left(\frac{\Re(z)}{\Im(z)}\right)\right).
\end{align*}
Hence, the derivative can be factored as
\begin{align*}
    f_{1}'(z) = (z-1)q(z),
\end{align*}
where
\begin{align*}
q(z) =     \frac{f_{1}(z)(R_{0}f_{1}(z)+1)}{z(R_{0}f_{1}(z)-1)(z+1)} \exp\left(\frac{2}{z+1} - 
    \frac{2}{f_{t}(z)}+ \int_{0}^{1}\frac{2}{(\phi_{s}(z)+s)^2} ds\right).
\end{align*}
Note that $\lim\limits_{z\to1} q(z)\neq 0$. Therefore, around $z=1$, the derivative behaves like $f_1'(z)\sim z-1$. Moreover, the limit $\lim\limits_{z\to 1}q'(z)$ exists and does not equal $\infty$.

Now we are in a position to show that the choice $R(t)=1/t$ produces a hull with a cusp at time $t=1$. Recall that $f(\partial\Delta)$ is said to have an \textit{inward-pointing cusp} at $f(1)\neq \infty$ if
\begin{align*}
    \arg(f(e^{i\vartheta}) - f(1)) \to\begin{cases}
        \beta \quad &\text{as }\vartheta\to0+,\\
        \beta +2\pi &\text{as }\vartheta \to 0-. 
    \end{cases}
\end{align*}
\begin{remark}
    Note that here \textit{inward-pointing} is understood from the perspective of the conformal map $f:\Delta\to D$, and it would be \textit{outward-pointing} from the perspective of a conformal map from the unit disk $\mathbb{D}$ to the inside of the hull.
\end{remark}
\begin{theorem}[\cite{Pommerenke}, Theorem 3.7]
Let $f:\Delta\to D$ be conformal and $\partial D$ be locally connected. Then    $f(\partial\Delta)$ has a cusp at $f(1)\neq \infty$ if and only if 
    \begin{align*}
        \lim\limits_{z\to 1}\arg\frac{f(z)-f(1)}{(z-1)^{2}}=\beta - \pi.
    \end{align*}
\end{theorem}
The Loewner map $f_{t}$ at time $t=1$ satisfies
\begin{align*}
    \lim\limits_{z\to1}\frac{f_{1}(z)-f_{1}(1)}{(z-1)^2} = \lim\limits_{z\to1}\frac{f_{1}'(z)}{2(z-1)} = \lim\limits_{z\to1}\frac{q(z)}{2} = \frac{q(1)}{2}\in\mathbb{R}.
\end{align*}
Thus
\begin{align*}
\lim\limits_{z\to1}\arg\frac{f_{1}(z)-f_{1}(1)}{(z-1)^2} = 0,
\end{align*}
and by continuity
\begin{align*}
\lim\limits_{\vartheta\to 0}\arg\frac{f_{1}(e^{i\vartheta})-f_{1}(1)}{(e^{i\vartheta}-1)^2} = 0.
\end{align*}
This immediately implies 
\begin{align*}
\arg \left(f_{1}(e^{i\vartheta})-f_{1}(1)\right) = \begin{cases}
        \pi &\vartheta\to 0+,\\
        \pi + 2\pi &\vartheta\to 0-.
    \end{cases}
\end{align*}
Therefore, the family $(f_{t}(\partial\mathbb{D}))_{t\in[0,1]}$ forms an inwards-pointing cusp at time $t=1$ at the point $f_{1}(1)$, and confirms the simulation in Figure \ref{fig:poisson_hull}.
 \begin{figure}[H]
    \centering
    \includegraphics[width=0.5\linewidth]{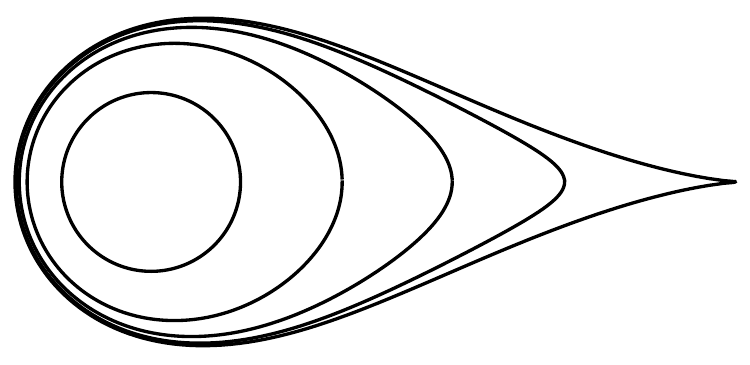}
    \caption{The hulls grown by the Poisson kernel density with $R(t) = 1/t$.}
    \label{fig:poisson_hull}
\end{figure}

\section{Minimizing the entropy}\label{section:problems}
In this section we discuss several optimization problems motivated by the problem to compute more explicitly the minimizing rate function $\mathcal{H}^*$ in Corollary~\ref{thm:potato}. We will also discuss the example showing that the equipotential measure does not in general give the minimal entropy.

\subsection{Loewner--Kufarev evolution with constraint}
\label{SS:energymin}
While not exactly the setting of Corollary~\ref{thm:potato}, the following formulation is convenient to work with.
\begin{problem}[Whole-plane Loewner--Kufarev entropy with prescribed interface]\label{problem:analytic} 
Let $\gamma$ be a Jordan curve separating $0$ from $\infty$ and assume that the capacity of $\gamma$ equals $e^T$. 
\begin{align}\label{eq:Problem:analytic} 
 \textrm{Minimize} \quad \mathcal{H}(\rho)= \int_{-\infty}^T h_{S^{1}}(\rho_t) dt, 
 \end{align}
 where $\rho \in \mathcal{N}$ is subject to the constraint that $\partial \Delta_T = \gamma$, where $\Delta_t = f_t(\Delta)$ and $(f_t)_{t\in \mathbb{R}}$ is the Loewner chain of $\rho$.
\end{problem}

The analog of Problem~\ref{problem:analytic} with entropy replaced by energy was solved in \cite{viklund_loewnerkufarev_2024} and a solution exists if and only if $\gamma$ is a Weil-Petersson quasicircle, i.e., has finite Loewner energy. Let $D$ be the bounded component of $\mathbb{C} \smallsetminus \gamma$. Consider the \emph{equipotential foliation} of $D$ defined as follows. Let $f: \mathbb{D} \to D$ be the conformal map such that $f(0)=0, f'(0)>0$. The images under $f$ of concentric circles form a foliation of $D$ which can be described by Loewner evolution. Similarly, we can use the exterior normalized conformal map $g$ to foliate the exterior $\mathbb{C} \smallsetminus D$ by equipotentials. Write $\rho^\gamma_\textrm{eq.}$ for the corresponding driving measure with marginals integrating to $1$. Note that $\rho^\gamma_{\textrm{eq.}} \equiv 1/2\pi d \theta dt$ for $t \ge 0$ (this uses that $\gamma$ has capacity $1$). 
\begin{theorem}[\cite{viklund_loewnerkufarev_2024}, Theorem~1.4]\label{VW-optimizer}
    We have
    \[
    \inf S(\rho) = S(\rho^\gamma_{\textrm{eq.}}) = \frac{1}{16}I^L(\gamma) - \frac{1}{8}\log \frac{|f'(0)|}{|g'(\infty)|},
    \]
    where the infimum is taken over $\rho \in \mathcal{N}$ generating $\gamma$ as a shape and $I^L$ is the Loewner energy of $\gamma$.
\end{theorem}
\begin{corollary}
The minimal entropy $\mathcal{H}^*(\gamma)$ needed to generate $\gamma$ satisfies
    \[    
    \mathcal{H}^*(\gamma) \le \frac{1}{8} I^L(\gamma) - \frac{1}{4}\log \frac{|f'(0)|}{|g'(\infty)|}.
    \]
\end{corollary}
\begin{proof}
    By Theorem~\ref{T:energy_entropy} we know that $\mathcal{H}(\rho) \le 2S(\rho)$, so the estimate follows immediately from this estimate combined with Theorem~\ref{VW-optimizer}.
\end{proof}
 Given the corollary, one may ask whether equipotentials provide the solution to Problem~\ref{problem:analytic} as well. This is however not the case.


\begin{proposition}\label{problem:easy-energy}Consider Problem~\ref{problem:analytic}. There exists a Jordan curve $\gamma$ such that
   \[ \mathcal{H}^*(\gamma) < \mathcal{H}(\rho^\gamma_{\textrm{eq.}}),\]
   where $\rho^\gamma_{\textrm{eq.}}$ is the driving measure of the equipotential foliation corresponding to $\gamma$.
\end{proposition}
\begin{proof}[Sketch of proof.]
Let $\epsilon > 0$ be very small and consider the measure $\rho^\epsilon$ defined by $\rho_t^\epsilon \equiv 1/2\pi$ for $t \le 0$ and 
\[
    \rho_t^\epsilon \equiv  \frac{1}{2\pi(1-2\epsilon)}\mathds{1}_{(\epsilon, 2\pi - \epsilon)}(\theta), \qquad 0 \le t \le 1.
\] 
The hull at time at time $1$ is a ``Pac-Man'' shaped compact set $K_\epsilon \supset \mathbb{D}$ whose boundary is a Jordan curve $\gamma_\epsilon$ containing the interval $\{e^{i\theta}: \, \theta \in [-\epsilon, \epsilon]\}$. Write $D_\epsilon$ for the simply connected interior of $K_\epsilon$. Note that $\gamma_\epsilon \cap [1, \infty) = \{1\}$ and, as $\epsilon \to 0$, the interiors $D_\epsilon$ converge in the Carath\'eordory sense towards the slit disc $D_0:= (e \mathbb{D})\smallsetminus [1,e]$. Since the  conformal maps from the unit disc converge uniformly on compacts, the equipotential foliations converge and so do the corresponding driving measures.

By Lemma~\ref{lem:entropy-facts} $\mathcal{H}$ is lower semi-continuous. Therefore, if we let $\rho_{\textrm{eq.}}^{D_\epsilon}$ be the driving measure of the equipotential foliation of $D_\epsilon$ considered up to time $1$,   
\[
\liminf_{\epsilon \to 0} \mathcal{H}(\rho_{\textrm{eq.}}^{D_\epsilon}) \ge  \mathcal{H}(\rho_{\textrm{eq.}}^{D_0}) > 0.
\]
On the other hand, we compute that $\mathcal{H}(\rho^\epsilon) =  O(\epsilon)$. It follows that if $\epsilon>0$ is sufficiently small then
\[
\mathcal{H}(\rho^\epsilon) < \mathcal{H}(\rho_{\textrm{eq.}}^{D_\epsilon}),
\]
as claimed.
\end{proof}

\subsection{Non-analytic Loewner evolution as a transport equation}

\label{SS:non-analytic}
We now consider a simpler optimization problem. It is convenient to work in the half-plane, \emph{chordal}, setting. 
First recall that the chordal Loewner--Kufarev PDE is
\begin{align}\label{eq:chordal-loewner-pde}
 \partial_t f_t(z) = - f'_t(z) \int_{\mathbb{R}} \frac{d\mu_t(x)}{z-x}, \qquad z \in \mathbb{H},
 \end{align}
 where $(\mu_t)_{t \ge 0}$ is usually assumed to be a family of probability measures on $\mathbb{R}$.
 Suppose that $(\rho_t)_{t \ge 0}$ is a measurable family of real densities $d\mu_t = \rho_t(x) dx$, smooth say, such that $\int_{\mathbb{R}} \rho_t(x) dx = 1$.
 For each $t$, let $P_t$ be the Poisson integral of $\rho_t$, that is, 
\[P_t(x,y) = -\frac{1}{\pi}\Im \int_\mathbb{R} \frac{ \rho_t(\xi) d\xi}{z-\xi}, \qquad z=x+iy. \]
For $x \in \mathbb{R}$, consider the following real PDE,
\begin{align}\label{eq:transport}
\partial_tu(t,y) = P_t(x,y) \partial_y u(t,y), \qquad u(0,y) = y.
\end{align}
This is a transport equation in non-conservative form, with velocity field $t \mapsto P_t(x,y)$, which depends on $\rho_t$ in a non-local manner. Note that $u$ depends implicitly on $x$.
 A similar equation was considered by Carleson and Makarov who interpreted it as a non-analytic simplification of the Loewner--Kufarev equation:
\eqref{eq:transport} arises up to a factor $\pi$ after taking the imaginary part of \eqref{eq:chordal-loewner-pde} (with $\mu_t$ replaced by $\rho_t$) and ignoring the $x$-dependence, see Section~1.5 of \cite{carleson2001aggregation}. In this context we view \eqref{eq:transport} as describing the evolution of a non-analytic map
\[
U_t: x+iy \mapsto x + i u^x(t,y), \quad U_0(x + iy) = x+iy,
\]
and the image of the unit interval describes an interface or ``shape''.

We are interested in the following optimization problem, where we recall the notation $h_\mathbb{R}(f) = \int_\mathbb{R} f \log f dx$. 
\begin{problem}[Non-analytic Loewner evolution]
    Let  $\gamma:\mathbb{R}\to \mathbb{R}_+$ be a smooth function.
 \begin{align}\label{eq:Problem:minimize} 
 \textrm{Minimize} \quad H(\rho)= \int_0^1 h_\mathbb{R}(\rho_t) dt, 
 \end{align}
 where $\rho$ is subject to the following constraints:
 \begin{itemize}
 \item{For each $t \in [0,1]$, 
  \[\int_\mathbb{R} \rho_t(x) dx =1,\]} 
  \item{$\rho$ generates $\gamma$ in the sense that for each $x \in \mathbb{R}$, $u(t,y) = u^x(t,y)$ satisfies \eqref{eq:transport}  with velocity field $P_t$ and 
  \[
\lim_{y \to 0+} u(1,y) = \gamma(x).
 \]}
 \end{itemize}
     \end{problem}
     In other words, we prescribe the interface $U_1(x)=\gamma(x)$ at time $1$ and ask for a measure with minimal entropy that produces this interface via \eqref{eq:transport}. This is a version of Problem~\ref{problem:analytic}. However, while the dynamics is simpler, solutions do not generate conformal maps which results in loss of a source of integrability. Moreover, even though \eqref{eq:transport} can be solved by characteristics via the ODE
     \[
     \dot{Y} = -P_t(x,Y), \qquad Y(0) = y_0,
     \]
     the non-linear dependence on $\rho$ still results in a complicated constraint. Even determining a useful condition for $\gamma$ to be generated by \eqref{eq:transport} with some $\rho$ appears non-trivial. 
     To make progress, we will further simplify \eqref{eq:transport} by replacing $P_t(x,y)$ by $\rho_t(x) = \lim_{y \to 0+}P_t(x,y)$ and for each $x \in \mathbb{R}$ we obtain the following transport equation,
\begin{align}\label{eq:transport2}
\partial_tv(t,y) = \rho_t(x) \partial_y v(t,y) , \qquad v(0,y) = y.
\end{align}
Since the right-hand side of \eqref{eq:transport2} can be written $\partial_t u = \partial_y(\rho_t v)$ the equation is conservative. (Recall that $x$ is a parameter and not a variable.) Now it is not hard to determine what $\gamma$ can be generated and we formulate the following version of the problem.
\begin{problem}[Conservative non-analytic Loewner evolution]\label{problem:linear}
    Let  $\gamma:\mathbb{R}\to \mathbb{R}_+$ be a smooth function.
 \begin{align}\label{eq:Problem:linear}
 \textrm{Minimize} \quad H(\rho)= \int_0^1 h_\mathbb{R}(\rho_t) dt, 
 \end{align}
 where $\rho$ is subject to the following constraints:
 \begin{itemize}
 \item{For each $t \in [0,1]$, 
  \[\int_{\mathbb{R}}\rho_t(x) dx =1,\]} 
  \item{$\rho$ generates $\gamma$ in the sense that for each $x \in \mathbb{R}$, $v(t,y) = v^x(t,y)$ satisfies \eqref{eq:transport2}  with velocity field $\rho$ and 
  \[
\lim_{y \to 0+} v(1,y) = \gamma(x).
 \]}
 \end{itemize}
     \end{problem}
\begin{proposition}
   A solution to Problem~\ref{problem:linear} exists if and only if $\gamma$ satisfies $\int_\mathbb{R} \gamma(x) dx = 1$ and $h_\mathbb{R}(\gamma) < \infty$, in which case the  solution is constant in $t$,
\[
\rho^{*}_t(x) \equiv \gamma(x).
\]
The minimal entropy equals the entropy of $\gamma$,
\[
H^* = h_\mathbb{R}(\gamma).
\]
\end{proposition}
\begin{proof}
    The solution to \eqref{eq:transport2} is 
\[
v(t,y) = y + \int_0^t \rho_t(x) dt.
\]
Therefore, the condition that $\rho$ generates $\gamma$ becomes
\[
 v(1,0) = \int_0^1 \rho_t(x) dt= \gamma(x)
\]
and so all constraints are linear. 
Combined with the requirement that $\int \rho_t dx = 1$ this gives the necessary condition $\int \gamma(x) dx = 1$. Consider now the action functional with Lagrange multipliers corresponding to the constraints \[
 L(\rho, \lambda, \mu)=H(\rho) +
	\int_0^1 \lambda(t)\left(\int_{-\infty}^\infty \rho_t(x) dx -1\right)dt
		+ \int_{-\infty}^\infty \mu(x)\left(\int_0^1 \rho_t(x) dt - \gamma(x)\right)dx.
\]
The Euler-Lagrange equation for $\rho$ is
\[
\log \rho + 1 + \lambda + \mu = 0,
\]
so a stationary solution can be written in the form
\[
\rho_t(x) = \exp\{- 1 - \lambda(t) - \mu(x)\} =: T(t)X(x).
\]
On the other hand, the condition $\int_\mathbb{R} \rho_t(x) dx = 1$ implies $T(t) \equiv 1$ and finally the condition   $\int_0^1 \rho_t(x) dt = \gamma(x)$ implies $X(x) = \gamma(x)$.
\end{proof}







\bibliographystyle{plain}

\bibliography{references}

@book{dembo2009large,
  title={{Large Deviations Techniques and Applications}},
  author={Dembo, Amir and Zeitouni, Ofer},
  year={2009},
  publisher={Springer}
}

@article{Bishop2025,
  author       = {Christopher J.\ Bishop},
  title        = {Weil–Petersson curves, $\beta$-numbers, and minimal surfaces},
  journal      = {Annals of Mathematics},
  volume       = {202},
  number       = {1},
  pages        = {111--188},
  year         = {2025},
  doi          = {10.4007/annals.2025.202.1.2},
}

@book{DurenHp,
  author    = {Duren, Peter L.},
  title     = {Theory of $H^p$ Spaces},
  series    = {Pure and Applied Mathematics},
  volume    = {38},
  publisher = {Academic Press},
  address   = {New York},
  year      = {1970},
}

@article{Wang2019_DeterministicLoewnerChainEnergy,
  author = {Wang, Yilin},
  title = {{The energy of a deterministic Loewner chain: Reversibility and interpretation via SLE$_{0+}$}},
  journal = {J. Eur. Math. Soc. (JEMS)},
  volume = {21},
  number = {7},
  pages = {1915--1941},
  year = {2019},
}

@article{Wang2019LoewnerEnergy,
  author = {Yilin Wang},
  title = {{Equivalent descriptions of the Loewner energy}},
  journal = {Inventiones Mathematicae},
  volume = {218},
  number = {2},
  pages = {573--621},
  year = {2019},
  doi = {10.1007/s00222-019-00887-0},
}

@article{GumenyukPrause2018,
  author  = {Gumenyuk, Pavel and Prause, Istv{\'a}n},
  title   = {Quasiconformal extensions, {L}oewner chains, and the $\lambda$-Lemma},
  journal = {Analysis and Mathematical Physics},
  volume  = {8},
  pages   = {621--635},
  year    = {2018},
  doi     = {10.1007/s13324-018-0247-3},
}

@article{Becker1972,
  author  = {Becker, J.},
  title   = {L\"ownersche Differentialgleichung und quasikonform fortsetzbare schlichte Funktionen},
  journal = {Journal f{\"u}r die reine und angewandte Mathematik},
  volume  = {255},
  year    = {1972},
  pages   = {23--43},
}

@book{Pommerenke,
  author    = {Pommerenke, Christian},
  title     = {Boundary Behaviour of Conformal Maps},
  series    = {Grundlehren der Mathematischen Wissenschaften},
  volume    = {299},
  publisher = {Springer-Verlag},
  address   = {Berlin},
  year      = {1992},
  isbn      = {3-540-54939-2},
}

@book{deuschel2001large,
  title={{Large Deviations}},
  author={Deuschel, Jean-Dominique and Stroock, Daniel W.},
  volume={342},
  year={2001},
  publisher={American Mathematical Soc.}
}

@article{puhalskii2006large,
  title={Large deviations for stochastic processes},
  author={Puhalskii, Anatolii Anatol'evich},
  journal={Notes for LMS/EPSRC Short Course: Stochastic Stability, Large Deviations and Coupling Methods, Heriot-Watt University, Edinburgh},
  year={2006}
}

@article{Wang2022survey,
  title={{Large deviations of Schramm-Loewner evolutions: A survey}},
  author={Wang, Yilin},
  journal={Probability Surveys},
  year={2022}
}

@article{hastings1998laplacian,
  title={Laplacian growth as one-dimensional turbulence},
  author={Hastings, Matthew B and Levitov, Leonid S},
  journal={Physica D: Nonlinear Phenomena},
  volume={116},
  number={1-2},
  pages={244--252},
  year={1998},
  publisher={Elsevier}
}

@book{bogachev2007measure,
  title={{Measure Theory}},
  author={Bogachev, Vladimir Igorevich},
  year={2007},
  publisher={Springer}
}

@article{carleson2001aggregation,
  title={{Aggregation in the plane and Loewner's equation}},
  author={Carleson, Lennart and Makarov, Nikolai},
  journal={Communications in Mathematical Physics},
  volume={216},
  pages={583--607},
  year={2001},
  publisher={Springer}
}

@article{norris2012hastings,
  title={{Hastings--Levitov aggregation in the small-particle limit}},
  author={Norris, James and Turner, Amanda},
  journal={Communications in Mathematical Physics},
  volume={316},
  number={3},
  pages={809--841},
  year={2012},
  publisher={Springer}
}

@article{silvestri2017fluctuation,
  title={{Fluctuation results for Hastings--Levitov planar growth}},
  author={Silvestri, Vittoria},
  journal={Probability Theory and Related Fields},
  volume={167},
  pages={417--460},
  year={2017},
  publisher={Springer}
}

@article{berger2022growth,
  title={{Growth of stationary Hastings--Levitov}},
  author={Berger, Noam and Procaccia, Eviatar B and Turner, Amanda},
  journal={The Annals of Applied Probability},
  volume={32},
  number={5},
  pages={3331--3360},
  year={2022},
  publisher={Institute of Mathematical Statistics}
}

@article{berestycki2021explosive,
  title={{Explosive growth for a constrained Hastings--Levitov aggregation model}},
  author={Berestycki, Nathana{\"e}l and Silvestri, Vittoria},
  journal={Arkiv f{\"o}r Matematik},
  volume={61},
  number={1},
  pages={41--66},
  year={2023},
}

@article{johansson2012scaling,
  title={{Scaling limits of anisotropic Hastings--Levitov clusters}},
  author={Johansson Viklund, Fredrik and Sola, Alan and Turner, Amanda},
  journal={Annales de l'IHP Probabilit{\'e}s et statistiques},
  volume={48},
  number={1},
  pages={235--257},
  year={2012}
}

@article{johansson2009rescaled,
  title={{Rescaled L{\'e}vy--Loewner hulls and random growth}},
  author={Johansson, Fredrik and Sola, Alan},
  journal={Bulletin des sciences mathematiques},
  volume={133},
  number={3},
  pages={238--256},
  year={2009},
  publisher={Elsevier}
}

@article{rohde2005some,
  title={{Some remarks on Laplacian growth}},
  author={Rohde, Steffen and Zinsmeister, Michel},
  journal={Topology and its Applications},
  volume={152},
  number={1-2},
  pages={26--43},
  year={2005},
  publisher={Elsevier}
}

@article{johansson2015small,
  title={Small-particle limits in a regularized Laplacian random growth model},
  author={Johansson Viklund, Fredrik and Sola, Alan and Turner, Amanda},
  journal={Communications in Mathematical Physics},
  volume={334},
  pages={331--366},
  year={2015},
  publisher={Springer}
}

@article{kozlov_fine-grained_2007,
	title = {Fine-grained and coarse-grained entropy in problems of statistical mechanics},
	volume = {151},
	number = {1},
	journal = {Theoretical and Mathematical Physics},
	author = {Kozlov, Valery Vasil'evich and Treshchev, Dmitry Valerevich},
	year = {2007},
	pages = {539--555},
}

@article{viklund_loewnerkufarev_2024,
	title = {The {Loewner}–{Kufarev} energy and foliations by {Weil}–{Petersson} quasicircles},
	volume = {128},
	number = {2},
	journal = {Proceedings of the London Mathematical Society},
	author = {Viklund, Fredrik and Wang, Yilin},
	year = {2024},
	pages = {e12582},
}

@article{emery_yukich_1987ASP,
  title={{A simple proof of the logarithmic Sobolev inequality on the circle}},
  author={{\'E}mery, Michel and Yukich, Joseph E},
  journal={S{\'e}minaire de probabilit{\'e}s de Strasbourg},
  volume={21},
  pages={173--175},
  year={1987}
}

@phdthesis{theodosiadis,
    author = {Eleftherios Theodosiadis},
    title ={{Geometry of multi-slit Loewner
chains and semigroups of finite
shift}} ,
    school = {Stockholm University} ,
    year = 2024
}

@article{APW,
author = {Ang, Morris and Park, Minjae and Wang, Yilin},
title = {{Large deviations of radial SLE$_{\infty}$}},
volume = {25},
journal = {Electronic Journal of Probability},
publisher = {Institute of Mathematical Statistics and Bernoulli Society},
pages = {1--13},
year = {2020}
}

@article{guskov2023large,
  title={{A large deviation principle for the Schramm--Loewner evolution in the uniform topology}},
  author={Guskov, Vladislav},
  journal={Annales Fennici Mathematici},
  volume={48},
  pages={389--410},
  year={2023}
}

@article{krusell2024,
  title={{The $\rho$-Loewner Energy: Large Deviations, Minimizers, and Alternative Descriptions}},
  author={Krusell, Ellen},
  journal={arXiv:2410.08969},
  year={2024}
}

@article{abuzaid2025large,
  title={{Large deviations of SLE(0+) variants in the capacity parameterization}},
  author={Abuzaid, Osama and Peltola, Eveliina},
  journal={arXiv:2503.02795},
  year={2025}
}

@article{sola2013elementary,
  title={{Elementary examples of Loewner chains generated by densities}},
  author={Sola, Alan},
  journal={Annales Universitatis Mariae Curie-Sk{\l}odowska, sectio A--Mathematica},
  volume={67},
  number={1},
  year={2013},
  publisher={Wydawnictwo Uniwersytetu Marii Curie-Sk{\l}odowskiej}
}
\end{document}